\documentclass[reqno]{amsart}

\usepackage{amsthm}
\usepackage[usenames,dvipsnames]{pstricks}
\usepackage{epsfig}
\usepackage{pst-grad} 
\usepackage{pst-plot} 
\usepackage{graphicx,subfigure}
\usepackage{amsmath,amssymb}

\usepackage{acronym}
\acrodef{BO}{{\sl Benjamin-Ono}}
\acrodef{rBO}{{\sl regularized Benjamin-Ono}}
\acrodef{rILW}{{\sl regularized Intermediate Long Wave}}
\acrodef{DSW}{{\sl Dispersive Shock Wave}}
\acrodef{DSWs}{{\sl Dispersive Shock Waves}}
\acrodef{ILW}{{\sl Intermediate Long Wave}}
\acrodef{CGN}{{\sl Conjugate Gradient-Newton}}
\acrodef{SW/SW}{{\sl Shallow water / Shallow water}}
\acrodef{B/B}{{\sl Boussinesq / Boussinesq}}

\makeatletter
\newcommand{\sech}{\mathop{\operator@font sech}}
\newcommand{\sign}{\mathop{\operator@font sign}}
\makeatother

\newtheorem{lemma}{Lemma}[section]
\newtheorem{theorem}{Theorem}[section]
\newtheorem{proposition}{Proposition}[section]

\newtheorem{remark}{Remark}[section]

\numberwithin{equation}{section}

\begin{document}

\title[]{Solitary-wave solutions of the fractional nonlinear Schr\"{o}dinger equation. I.  Existence and numerical generation}


\author{Angel Dur\'an}
\address{\textbf{A.~Dur\'an:} Applied Mathematics Department, University of Valladolid, P/ Belen 15, 47011, Valladolid, Spain}
\email{angel@mac.uva.es}

\author{Nuria Reguera}
\address{\textbf{N.~Reguera:} Department of Mathematics and Computation, University of Burgos, 09001 Burgos, Spain}
\email{nreguera@ubu.es}







\begin{abstract}
The present paper is the first part of a project devoted to the fractional nonlinear Schr\"{o}dinger (fNLS) equation. It is concerned with the existence and numerical generation of the solitary-wave solutions. For the first point, some conserved quantities of the problem are used to search for solitary-wave solutions as relative equilibria. From the relative equilibrium condition, a result of existence via the Concentration-Compactness theory is derived. Several properties of the waves, such as the regularity and the asymptotic decay in some cases, are derived from the existence result. Some other properties, such as the monotone behaviour and the speed-amplitude relation, will be explored computationally. To this end, a numerical procedure for the generation of the profiles is proposed. The method is based on a Fourier pseudospectral approximation of the differential system for the profiles and the use of Petviashvili's iteration with extrapolation.
\end{abstract}

\maketitle
\tableofcontents

\section{Introduction}\label{sec1}
The present paper is the first part of a project concerning the fractional nonlinear Schr\"{o}dinger (fNLS) equation in its one-dimensional version. Considered here are the problems of existence of solitary wave solutions, their numerical generation and some of their mathematical properties. The second part of the project is planned to study the numerical approximation of the corresponding periodic initial-value poblem (ivp) and its application to analyze the dynamics of the solitary wave solutions by computational means, \cite{NRAD2}.

The fNLS equation has the form
\begin{eqnarray}
iu_{t}-\beta (-\Delta)^{s}u-\gamma |u|^{2\sigma}u=0,\quad {\bf x}\in \mathbb{R}^{n}, n\geq 1,\quad t>0.\label{fnls1}
\end{eqnarray} 
In (\ref{fnls1}), $\sigma>0$ determines the nonlinearity (which is cubic when $\sigma=1$) and $\Delta$ denotes the Laplace operator (with $\Delta=\partial_{xx}$ in the one-dimensional case), which here is presented in \lq fractional\rq\ way determined by the parameter $0<s<1$. This term shows the nonlocal character of (\ref{fnls1}), generalizing somehow the classical nonlinear Schr\"{o}dinger equation (NLS) (which would correspond to (\ref{fnls1}) with $s=1$), a classical model for the propagation of weakly nonlinear waves in dispersive media (see e.~g. \cite{SulemS1999}). The fractional Laplacian $(-\Delta)^{s}$ has a Fourier representation of the form
\begin{eqnarray*}
\widehat{(-\Delta)^{s}f}(\xi)=|\xi|^{2s}\widehat{f}(\xi),\quad \xi\in\mathbb{R}^{n},\label{fnls2}
\end{eqnarray*}
where
\begin{eqnarray*}
\widehat{f}(\xi)=\int_{\mathbb{R}^{n}}f({\bf x})e^{-i\xi \cdot {\bf x}}d{\bf x},
\end{eqnarray*}
is the Fourier transform of $f$ with the dot in the integrand standing for the Euclidean inner product in $\mathbb{R}^{n}$. The sign of $\beta\neq 0$ and $\gamma\neq 0$ determines the corresponding focusing and defocusing cases, as in the classical NLS. Following \cite{AblowitzP}, we may write (\ref{fnls1}) in the form
\begin{eqnarray}
\frac{i}{\beta}u_{t}-(-\Delta)^{s}u-\frac{\gamma}{\beta} |u|^{2\sigma}u=0,\quad {\bf x}\in \mathbb{R}^{n}, n\geq 1,\quad t>0,\label{fnls1bb}
\end{eqnarray} 
and the change $t\mapsto \beta t$, in the sense of writing $u(x,t)=v(x,\beta t)$, transforms (\ref{fnls1bb}) into
\begin{eqnarray}
{i}v_{t}-(-\Delta)^{s}v-\widetilde{\gamma} |v|^{2\sigma}v=0,\quad {\bf x}\in \mathbb{R}^{n}, n\geq 1,\quad t>0,\label{fnls1c}
\end{eqnarray} 
where $\widetilde{\gamma}=\frac{\gamma}{\beta}$. With these notations, $\widetilde{\gamma}>0$ is the defocusing case while $\widetilde{\gamma}<0$ is the focusing case. An alternative change $t\mapsto -t$ from (\ref{fnls1c}) leads to
\begin{eqnarray*}
{i}v_{t}+(-\Delta)^{s}v+\widetilde{\gamma} |v|^{2\sigma}v=0,\quad {\bf x}\in \mathbb{R}^{n}, n\geq 1,\quad t>0,
\end{eqnarray*} 
where now  $\widetilde{\gamma}>0$ is the focusing case and $\widetilde{\gamma}<0$ is the defocusing case. 

Equation (\ref{fnls1}) appears in the mathematical modeling of two main areas:
\begin{itemize}
\item It was originally introduced in Quantum Mechanics, \cite{Laskin2000,Laskin2002,Laskin2011,FrohlichJL2007}.
\item For some particular values of the parameters, (\ref{fnls1}) also appears in surface water wave models, \cite{ObrechtS2015,IonescuP2014}.
\end{itemize}
Summarized here are some mathematical properties of (\ref{fnls1}), see \cite{KleinSM2014} for details.
\subsubsection*{Well-posedness} The question of local/global well-posedness of the corresponding initial-value problem (ivp) for (\ref{fnls1}) is analyzed in \cite{ChoHHO2013,ChoHKL2015,ChoHKL2015b,GuoH2013,GuoSWZ2013,HongS2015}. The main results are the following. In \cite{ChoHKL2015}, Cho et al. study local well- and ill-posedness of the one-dimensional version of (\ref{fnls1}) with $\beta=-1, 1/2<s<1, \sigma=1$, in certain Sobolev spaces. These results are extended in \cite{HongS2015} to higher dimensions, $0<s<1, s\neq 1/2$ and $\sigma\geq 1$. For other types of nonlinearites (Hartree-type) well-posedness and blow-up are investigated in \cite{ChoHHO2013,ChoHKL2015b}. Global well-posedness results in the energy space are obtained in \cite{GuoSWZ2013} while blow-up phenomena, in related cases, are also studied in \cite{GuoH2013}.
\subsubsection*{Conserved quantities and symmetries}
The following quantities (mass and energy, resp.) are conserved by smooth and decaying enough solutions of (\ref{fnls1}): 
\begin{eqnarray}
M(u)&=&\int_{\mathbb{R}^{n}} |u({\bf x},t)|^{2} d{\bf x}, \label{fnls3}\\
E(u)&=&\int_{\mathbb{R}^{n}}\left(\frac{\beta}{2} |\nabla^{s}u({\bf x},t)|^{2}+\frac{\gamma}{2\sigma+2} |u({\bf x},t)|^{2\sigma+2}\right) d{\bf x},\label{fnls4}
\end{eqnarray}
where $\nabla^{s}$ is the operator with Fourier symbol
$$\widehat{\nabla^{s}f}(\xi)=(-i|\xi|)^{s}\widehat{f}(\xi),\quad \xi\in\mathbb{R}^{n}.$$
As far as the symmetries of (\ref{fnls1}) are concerned, we have the invariance under the scaling transformation
$$u({\bf x},t)\mapsto u_{\lambda}({\bf x},t)=\lambda^{s/\sigma}u(\lambda{\bf x},\lambda^{2s}t),\quad \lambda>0,
$$ in the sense that if $u$ is a solution of (\ref{fnls1}) then $u_{\lambda}$ is also a solution. Under this scaling, \cite{KleinSM2014}, if
$\dot{H}^{p}$ denotes the homogeneous Sobolev space of order $p$ (this consists of $\sigma$ times weakly differentiable functions $u$ such that $D^{\alpha}u\in L^{2}$ for $|\alpha|=\sigma$, see \cite{Adams1975}) then
\begin{eqnarray*}
||u_{\lambda}||_{\dot{H}^{p}}=||\nabla^{p}u_{\lambda}||_{L^{2}}=\lambda^{n/2-p-s/\sigma}||u||_{\dot{H}^{p}},
\end{eqnarray*}
where $||\cdot ||_{\dot{H}^{p}}$ denotes the corresponding norm in $\dot{H}^{p}$.
The equation (\ref{fnls1}) is called $\dot{H}^{p}$ critical when this scaling leaves the norm invariant, that is
$$\frac{n}{2}-\frac{s}{\sigma}=p,$$ ($p+s/\sigma=1/2$ when $n=1$). For $p=0$ we obtain the $L^{2}$ or mass critical case, whenever the dispersion rate is $s=s^{*}=\sigma n/2$ ($s^{*}=\sigma/2$ in the 1D case). The equation is called mass subcritical (resp. supercritical) when $s<s^{*}$ (resp. $s>s^{*}$). As in the case of the NLS, this can be used to study well-posedness and blow-up, \cite{Lenzmann2007}. Similar critical cases can be defined using the second invariant (\ref{fnls4}). Thus, the energy critical case for $s=p$ holds when the kinetic energy of the solution is a scale invariant quantity of the evolution. This yields another critical index $s_{*}=n\sigma/(2\sigma+2)$, which is equivalent to $\sigma=\sigma_{*}$ where 
$$\sigma_{*}(s,n)=\left\{\begin{matrix} \frac{2s}{n-2s}&0<s<\frac{n}{2}\\
+\infty & s\geq \frac{n}{2}\end{matrix}\right.$$
The study of the energy critical and supercritical cases reveals differences with the NLS equation, see \cite{KleinSM2014}.
\subsubsection*{Special solutions}
Note first that (\ref{fnls1}) admits plane-wave solutions
\begin{eqnarray*}
u({\bf x},t)=A e^{i{(\bf k}\cdot {\bf x}-\omega t)},
\end{eqnarray*}
when the following dispersion relation is satisfied
\begin{eqnarray}
\omega-\beta |{\bf k} |^{2s}-\gamma A^{2\sigma}=0.\label{fnls6}
\end{eqnarray}
The stability and dynamics of these plane wave solutions look to be different from those in the standard NLS (see \cite{DuoZ2016,DuoLZ} and references therein). 

A second group of solutions for the focusing case ($\gamma=-1$) is given by the standing wave solutions
\begin{eqnarray}
u({\bf x},t)=\varphi({\bf x})e^{i\omega t},\quad \omega\in\mathbb{R},\label{fnls7a}
\end{eqnarray}
where
\begin{eqnarray*}
|\varphi |^{2\sigma}\varphi-\beta (-\Delta)^{s}\varphi=\omega \varphi.\label{fnls7b}
\end{eqnarray*}
They can be seen as critical points of the energy subjected to a fixed value of the mass, with Fourier multiplier given by $\omega$. By rescaling $\varphi_{\omega}=\omega^{1/2\sigma}\varphi_{1}\left(\omega^{1/2s}x\right)$ one can assume $\omega=1$ and therefore $\varphi=\varphi_{1}$ solves
\begin{eqnarray}
|\varphi |^{2\sigma}\varphi=\beta (-\Delta)^{s}\varphi+ \varphi.\label{fnls7c}
\end{eqnarray}
Solution $\varphi\in H^{s}(\mathbb{R}^{n})\bigcap L^{2\sigma+2}(\mathbb{R}^{n})$ of (\ref{fnls7c}) are known to exist for $0<\sigma<\sigma_{*}$, while for $\sigma\geq \sigma_{*}$ (\ref{fnls7c}) does not admit any non trivial solution in $H^{s}(\mathbb{R}^{n})\cap L^{2\sigma+2}(\mathbb{R}^{n})$, \cite{KleinSM2014}. The so-called ground states are solutions $Q$ of (\ref{fnls7c}) with minimal energy, they are known to be real, radially symmetric and satisfy
\begin{eqnarray}
\beta (-\Delta)^{s}Q+ Q=Q^{2\sigma+1}.\label{fnls7d}
\end{eqnarray}
In addition, they decay algebraically, as $|x|\rightarrow\infty$, like $|{\bf x}|^{-(n+2s)}$ (cf. the case of the NLS, whose ground state solutions decay exponentially) and no explicit form for $Q$ is known. A technique to contruct numerically fractional ground states is described in \cite{KleinSM2014} and applied to study their stability by  computational means. Additional aspects of existence and orbital stability of these waves are studied, among others, in \cite{ChoHHO2014,FrankL,GuoH}.

In \cite{HongS2017}, Hong and Sire construct, for the one-dimensional, cubic case, a class of traveling \lq soliton\rq\ solutions of the form
\begin{eqnarray}
u(x,t)=e^{-it(|k|^{2s}-\omega^{2s})}Q_{\omega,k}(x-2ts|k|^{2s-2}k),\label{fnls8a}
\end{eqnarray}
with speed $c=2k\in \mathbb{R}, s\in (1/2,1)$ and $\omega>0$,  by using variational theory. The profile $Q=Q_{\omega,k}$ is a solution of
\begin{eqnarray*}
P_{k}Q+\omega^{2s}Q-|Q|^{2}Q=0,\label{fnls8b}
\end{eqnarray*}
where $P_{k}$ is the linear operator with Fourier representation
\begin{eqnarray*}
\widehat{P_{k}v}(\xi)=\left(|\xi+k|^{2s}-|k|^{2s}+2sk|k|^{2s-2}\xi\right)\widehat{v}(\xi),\quad \xi\in \mathbb{R}.
\end{eqnarray*}
The profile $Q=Q_{\omega,k}$ is obtained as a minimizer of certain functional and actually $Q=Q_{\omega,k}\in C^{\infty}(\mathbb{R})$. See also \cite{ABS1997} for a possible, alternative derivation. 

From the point of view of the numerical approximation, the nonlocal character of (\ref{fnls1}) implies the natural choice, made by some authors, of Fourier spectral methods to discretize in space the corresponding periodic initial-value problem. Thus, Kirkpatrick and Zhang, \cite{KirkpatrickZ2016}, base on these methods their numerical scheme to analyze, by computational means, some phenomena of decomposition of the coherent structure of standing wave solutions of equations of fNLS type but with a potential, as well as the possible turbulence formation. For the time integration, the authors propose  a second-order splitting method (which is, by the way, a typical choice as well). On the other hand, Klein et al., \cite{KleinSM2014}, also construct a numerical scheme with pseudospectral discretization to study computationally some aspects of the dynamics of (\ref{fnls1}): blow-up and its properties, stability of standing wave solutions (generated numerically with Newton-Krylov methods) and the dynamics of solutions at moderate times. In this case, the time integration is performed with different fourth-order methods of linearly implicit Runge-Kutta type for the focusing case and of splitting type for the defocusing case, \cite{Driscoll2002,Klein2008}. Also, Duo and Zhang, \cite{DuoZ2016}, construct three schemes with Fourier spectral spatial discretization, along with three time integrators of split-step, Crank-Nicolson and relaxed type. The three methods preserve the corresponding discrete version of the mass quantity (\ref{fnls3}) and time reversibility. Some differences appear in the preservation of the energy (\ref{fnls4}) and the dispersion relation (\ref{fnls6}). Their efficiency (in particular, the second order of convergence) is illustrated numerically. Additionally, Antoine et al., \cite{AntoineTZ2016}, introduce several numerical strategies to generate approximations to standing waves and to the dynamics of relate equations of fNLS type, which include a rotational term and nonlocal interactions. In the case of the dynamics, after reformulating the equation to rule out the presence of the rotational term, a time splitting, pseudospectral scheme is applied. Finally, as an alternative to the spectral approach, Wang and Huang, \cite{WangH2015}, introduce a finite difference (FD) scheme of Crank-Nicolson type which preserves the discrete mass and energy. The authors prove the existence of numerical solution and the convergence of the method under suitable conditions on time and space stepsizes. (See also \cite{WangXY2014} for a FD discretization of a coupled system of FNLS equations, with Dirichlet boundary conditions.) The use of finite element methods, for the more general fractional Ginzburg-Landau equation, is treated in \cite{LiHW2017}.

Considered here is the one-dimensional version of (\ref{fnls1}) with $\beta>0$ and $\gamma<0$. For simplicity, $\beta=1, \gamma=-1$ will be taken, so that (\ref{fnls1}) for $n=1$ takes the form
\begin{eqnarray}
iu_{t}- (-\partial_{xx})^{s}u+ |u|^{2\sigma}u=0,\quad {x}\in \mathbb{R},\quad t>0.\label{fnls1d}
\end{eqnarray} 
In (\ref{fnls1d}), $\sigma>0, 0<s<1$. We are interested in several features of (\ref{fnls1d}) and the present paper is focused on the existence of solitary wave solutions. The main contributions, summarized here are the following:
%
%
%
\begin{itemize}
\item We adapt the interpretation of solitary wave solutions of the classical NLS considered in \cite{DuranS2000} , in terms of symmetry groups and relative equilibria (RE). For the fractional case (\ref{fnls1d}), we first derive a new conserved quantity (momentum) of the ivp and study those equilibria of the energy restricted to fixed values of the mass and momentum. A new family of solitary waves is then obtained from the application of a suitable symmetry group determined by the mass and momentum quantities. The existence and regularity are derived from the application of the Concentration-Compactness theory, \cite{Lions}. It may be worth pointing out some remarks here. The derivation of the waves is valid for $s\in (1/2,1]$, including the classical NLS ($s=1$). In that case, the well known soliton-type solutions, \cite{DuranS2000}, are obtained as a subfamily of the waves derived here, from a particular choice of the phase. The same subfamily for the fractional case $s\in (1/2,1)$ corresponds to the waves introduced in \cite{HongS2017}. 

As the (complex) waves are considered as relative equilibria with respect to two quantities, the conditional variational problem involves two Lagrange multipliers. They respectively determine phase and speed of the complex profiles. For a fixed value of the phase, there is a limiting speed for which the existence of solitary waves is ensured for velocities within the range limited by this speed of sound. 

\item In addition, most of the information about the structure and properties of these solitary wave solutions is obtained here from a numerical study of generation of approximate profiles. 
In this sense, Klein et al., \cite{KleinSM2014}, generate numerically fractional ground states (see also references therein) by approximating the profile $Q$ in (\ref{fnls7d}) with discrete Fourier series and solving the resulting algebraic equations for the discrete coefficients with Newton type iteration methods. The first point under study is then searching for efficient numerical approximations to the solitary wave profiles by using discrete Fourier techniques as well, but with different iterative techniques. Specifically, our proposal consists of using methods of Petviashvili type combined with acceleration techniques based on extrapolations. This has been successfully applied in other cases, e.~g. \cite{AlvarezD2014,AlvarezD2016}, and its performance here will serve us to study several features of the waves: asymptotic decay, speed-amplitude relation, even character, and monotone decay. 

\end{itemize}
The structure of the paper is as follows. Section \ref{sec2} is devoted to the existence of solitary waves. The formulation as relative equilibrium solutions requires the introduction of some ideas on symmetry groups, conserved quantities and the relative equilibrium condition. This is written in a proper way to prove the existence via the Concentration Compactness theory. In section  \ref{sec3} some additional properties of the waves, mentioned above, are investigated. For a particular choice of the phase function, some theoretical results can be proved. For the general case, the numerical procedure for the generation of approximate profiles is first introduced and validated, and then used to develop the corresponding computational study. Some concluding remarks are made in section \ref{sec4}.

The following notation will be used throughout the paper. The $L^{2}$-inner product will be denoted by $(\cdot,\cdot)$. 
For $v_{j}, w_{j}\in L^{2}, j=1,2$, the inner product in $L^{2}\times L^{2}$ is
\begin{eqnarray}
\langle \begin{pmatrix}v_{1}\\w_{1}\end{pmatrix},\begin{pmatrix}v_{2}\\w_{2}\end{pmatrix}\rangle=\int_{\mathbb{R}}(v_{1}v_{2}+w_{1}w_{2})dx,\label{fnls_234b}
\end{eqnarray}
On the other hand, for the sake of simplicity and since no confusion is possible, the norm of the $L^{p}=L^{p}(\mathbb{R})$ space and of the product $L^{p}\times L^{p}$ will be denoted by $||\cdot||_{L^{p}}$. In a similar way, for $s\geq 0$, $||\cdot||_{s}$ will stand for the norm of the $L^{2}$-based Sobolev space $H^{s}=H^{s}(\mathbb{R})$ and of the product space $H^{s}\times H^{s}$.

For a linear operator $\mathcal{L}$, the Calderon commutator $[\mathcal{L},\cdot]\cdot$ is defined as 
\begin{eqnarray}
[\mathcal{L},f]g=\mathcal{L}(fg)-f\mathcal{L}g.\label{Calc}
\end{eqnarray}
Throughout the paper, $C$ will denote a positive constant that may depend on fixed parameters.
\section{Existence of solitary waves}\label{sec2}
In this section the existence of solitary wave solutions of (\ref{fnls1d}) is analyzed. Note first that we can alternative write (\ref{fnls1d}) as a system
\begin{eqnarray}
v_{t}-(-\partial_{xx})^{s}w+(v^{2}+w^{2})^{\sigma}w&=&0,\nonumber\\
-w_{t}-(-\partial_{xx})^{s}v+(v^{2}+w^{2})^{\sigma}v&=&0.\label{fnls1b}
\end{eqnarray}
for $u=v+iw$.

\subsection{Symmetry groups and conserved quantities}
\label{sec21}
\begin{lemma}
\label{lemmaNR1}
 The following quantities (mass, momentum, energy) are conserved by smooth and decaying enough solutions of  (\ref{fnls1d}): 
\begin{eqnarray}
I_{1}(v,w)&=&\frac{1}{2}\int_{\mathbb{R}}(v^{2}+w^{2})dx=\frac{1}{2}\int_{\mathbb{R}}|u|^{2}dx, \label{fnls3a}\\
I_{2}(v,w)&=&\frac{1}{2}\int_{\mathbb{R}}(vw_{x}-wv_{x})dx=\frac{1}{2}\int_{\mathbb{R}}{\rm Im}(u\overline{u}_{x})dx, \label{fnls3b}\\
H(v,w)&=&\int_{\mathbb{R}}\left(\frac{1}{2}\left( (|D|^{s}v)^{2}+(|D|^{s}w)^{2}\right)-\frac{1}{2\sigma+2} (v^{2}+w^{2})^{\sigma+1}\right) d{x},\label{fnls3c}
\end{eqnarray}
where $u=v+iw$ and $|D|^{s}$ is defined as
$$\widehat{|D|^{s}f}(\xi)=|\xi|^{s}\, \widehat{f}(\xi),\quad \xi\in\mathbb{R}.$$
\end{lemma}
\begin{proof}
The quantities (\ref{fnls3a}) and (\ref{fnls3c}) are the 1D versions of (\ref{fnls3}) and (\ref{fnls4}) respectively. As far as (\ref{fnls3b}) is concerned, if $u=v+iw$ is a solution of (\ref{fnls1b}), we have
\begin{eqnarray*}
\frac{d}{dt}I_{2}(v,w)&=&I_{21}+I_{22},\\
I_{21}&=&\frac{1}{2}\int_{\mathbb{R}}\left(w_{x}(-\partial_{xx})^{s}w-w(-\partial_{xx})^{s}w_{x}+
v_{x}(-\partial_{xx})^{s}v-v(-\partial_{xx})^{s}v_{x}\right)dx,\\
I_{22}&=&\frac{1}{2}\int_{\mathbb{R}}\left(-ww_{x}(v^{2}+w^{2})^{\sigma}+w\partial_{x}\left(w(v^{2}+w^{2})^{\sigma}\right)\right.\\
&&\left.-vv_{x}(v^{2}+w^{2})^{\sigma}+v\partial_{x}\left(v(v^{2}+w^{2})^{\sigma}\right)\right)dx.
\end{eqnarray*}
The second integral can be written as
\begin{eqnarray*}
I_{22}=\frac{1}{2}\int_{\mathbb{R}}(v^{2}+w^{2})\partial_{x}(v^{2}+w^{2})^{\sigma}dx=\frac{\sigma}{2\sigma+2}\int_{\mathbb{R}}\partial_{x}\left({(v^{2}+w^{2})^{\sigma+1}}\right)dx,
\end{eqnarray*}
which vanishes if $v,w\rightarrow 0$ as $x\rightarrow\pm\infty$. On the other hand, we can make use of Plancherel's identity to obtain
\begin{eqnarray*}
I_{21}&=&\int_{\mathbb{R}}\left((i\xi)\widehat{w}(\xi)|\xi|^{2s}\overline{\widehat{w}(\xi)}+
(i\xi)\widehat{v}(\xi)|\xi|^{2s}\overline{\widehat{v}(\xi)}\right)dx\\
&=&\int_{\mathbb{R}}\left((i\xi)|\xi|^{2s}|\widehat{w}(\xi)|^{2}+
(i\xi)|\xi|^{2s}|\widehat{v}(\xi)|^{2}\right)dx=0.
\end{eqnarray*}
\end{proof}
We note that the new quantity (\ref{fnls3b}) will be also called momentum, in accordance with the case of the classical NLS.
The r\^{o}le of the three invariants (\ref{fnls3a})-(\ref{fnls3c})  in the generation of solitary-wave solutions of (\ref{fnls1}) is two-fold. The first one is concerned with the symmetry groups of (\ref{fnls1}), \cite{Olver}, in the classical sense of groups of transformations taking solutions into solutions. The following result (whose proof is direct) clarifies this point.
\begin{lemma}
\label{lemmaNR2}
 Let $(v,w)\mapsto G_{(\alpha,\beta)}(v,w), \alpha,\beta\in\mathbb{R}$ the transformation defined as
\begin{eqnarray}
G_{(\alpha,\beta)}(v,w)(x)=\begin{pmatrix} \cos\alpha&-\sin\alpha\\\sin\alpha&\cos\alpha\end{pmatrix}\begin{pmatrix}v(x-\beta)\\w(x-\beta)\end{pmatrix}.\label{symg}
\end{eqnarray}
Let $(v(x,t),w(x,t))$ be a solution of (\ref{fnls1b}). Then $(\widetilde{v},\widetilde{w})= G_{(\alpha,\beta)}(v,w)$ is a solution of (\ref{fnls1b}).
\end{lemma}
As it is well known, \cite{Olver}, the relation between (\ref{symg}) and (\ref{fnls3a}), (\ref{fnls3b}) is given by the property that the invariants determine the infinitesimal generators of the symmetry group, in the sense that for all $x\in\mathbb{R}$
\begin{eqnarray*}
\frac{d}{d\alpha}\Big|_{\alpha=0}G_{(\alpha,0)}(v,w)(x)&=&\begin{pmatrix}0&1\\-1&0\end{pmatrix}\delta I_{1}(v,w)(x),\\
\frac{d}{d\beta}\Big|_{\beta=0}G_{(0,\beta)}(v,w)(x)&=&\begin{pmatrix}0&1\\-1&0\end{pmatrix}\delta I_{2}(v,w)(x),
\end{eqnarray*}
where $\delta I$ denotes variational (Fr\'echet) derivative of $I$
$$\delta I(v,w)=\left(\frac{\delta I}{\delta v},\frac{\delta I}{\delta w}\right)^{T}.$$

\subsection{Relative equilibria}
\label{sec22}
We note that lemmas \ref{lemmaNR1} and \ref{lemmaNR2} also hold in the classical NLS equation (limiting case $s=1$ in (\ref{fnls1})), \cite{DuranS2000}. Motivated by this, we may go one step further and study the existence of relative equilibrium solutions of (\ref{fnls1}). This means that we look for profiles $u_{0}=v_{0}+iw_{0}$ which are critical points of the Hamiltonian at fixed values of the mass and momentum quantities
\begin{eqnarray*}
\delta\left(H(u_{0})-\lambda_{0}^{1}I_{1}(u_{0})-\lambda_{0}^{2}I_{2}(u_{0})\right)&=&0,\\
I_{i}(u_{0})&=&c_{i},i=1,2,
\end{eqnarray*}
for real $\lambda_{0}^{j}, j=1,2$ and some $c_{i}, i=1,2$.
The relative equilibrium (RE) equation will take the form
\begin{eqnarray}
-(-\partial_{xx})^{s}u_{0}+|u_{0}|^{2\sigma}u_{0}-\lambda_{0}^{1}u_{0}-i\lambda_{0}^{2}\partial_{x}u_{0}=0,\label{fnls22_1}
\end{eqnarray}
which can be written as a real system
\begin{eqnarray}
-(-\partial_{xx})^{s}v_{0}-\lambda_{0}^{1}v_{0}+\lambda_{0}^{2}w_{0}'+(v_{0}^{2}+w_{0}^{2})^{\sigma}v_{0}&=&0,\nonumber\\
-(-\partial_{xx})^{s}w_{0}-\lambda_{0}^{1}w_{0}-\lambda_{0}^{2}v_{0}'+(v_{0}^{2}+w_{0}^{2})^{\sigma}w_{0}&=&0.\label{fnls22_2}
\end{eqnarray}
The existence of solutions $u_{0}=(v_{0},w_{0})$ of (\ref{fnls22_2}) will be analyzed in section \ref{sec23} for $s\in (1/2,1]$ and $\sigma>0$. 
If we write $u_{0}(x)=e^{i\theta(x)}\rho(x)$ with real $\rho$, and from
$$\varphi(x,x_{0},\theta_{0})=G_{(\theta_{0},x_{0})}(u_{0})=\rho(x-x_{0})e^{i\theta(x-x_{0})+i\theta_{0}},$$ which represents the elements of the {\em orbit} through $u_{0}=(v_{0},w_{0})$ by the symmetry group (\ref{symg}), 
a two-parameter family of solitary-wave solutions of (\ref{fnls1}) will take the form 
\begin{eqnarray}
\psi(x,t,a,c,x_{0},\theta_{0})&=&G_{(t\lambda_{0}^{1},t\lambda_{0}^{2})}(\varphi)\nonumber\\
&=&\rho(x-t\lambda_{0}^{2}-x_{0})e^{i(\theta(x-t\lambda_{0}^{2}-x_{0})+\theta_{0}+\lambda_{0}^{1}t)}.\label{fnls22_7}
\end{eqnarray}
A particular subfamily of (\ref{fnls22_7}) can be emphasized. Taking 
$\theta(x)=Ax$ for constant $A$, then it holds that $\rho$ satisfies
\begin{eqnarray}
R_{A}\rho=\rho^{2\sigma+1},\label{fnls22_3}
\end{eqnarray}
where $R_{A}$ is the Fourier multiplier operator with Fourier symbol
\begin{eqnarray}
\widehat{R_{A}f}(\xi)&=&r(\xi)\widehat{f}(\xi),\; r(\xi)=|\xi+A|^{2s}-\lambda_{0}^{2}\xi+\underbrace{\lambda_{0}^{1}-A\lambda_{0}^{2}}_{B},\quad \xi\in\mathbb{R}.\label{fnls22_4}
\end{eqnarray}
Formula (\ref{fnls22_4}) extends the arguments made in  \cite{DuranS2000} for the classical NLS to the fractional case: when $s=1$, the Fourier symbol is
\begin{eqnarray}
r(\xi)=\xi^{2}+(2A-\lambda_{0}^{2})\xi+\lambda_{0}^{1}+A^{2}-\lambda_{0}^{2}A,\quad \xi\in\mathbb{R},\label{fnls22_4a}
\end{eqnarray} 
and the soliton solutions of the classical NLS can be obtained from those relative RE solutions corresponding to taking $A=\lambda_{0}^{2}/2$ and then, cf. \cite{DuranS2000}
$$r(\xi)=\xi^{2}+a,\quad a=\lambda_{0}^{1}-\frac{(\lambda_{0}^{2})^{2}}{4}.$$
Our main difference here is the presence of a nonlocal operator in the equation for $\rho$. We write (\ref{fnls22_3}) in the form
\begin{eqnarray}
(M+a)\rho=\rho^{2\sigma+1},\label{fnls22_5}
\end{eqnarray}
with $a=B+|A|^{2s}$, $M$ with Fourier symbol
\begin{eqnarray}
m(\xi)=|\xi+A|^{2s}-\lambda_{0}^{2}\xi-|A|^{2s},\quad \xi\in\mathbb{R},\quad 0<s<1.\label{fnls22_6}
\end{eqnarray}
The following properties of $m$ are proved in \cite{HongS2017}.
\begin{lemma}
\label{lemmaNR3}
The symbol $m$ in (\ref{fnls22_6}) satisfies:
\begin{enumerate}
\item $m(0)=0$.
\item $m'(0)=0\Leftrightarrow \lambda_{0}^{2}=2s|A|^{2s-2}A$.
\item $m''(\xi)>0$ if $s>1/2$.
\end{enumerate}
\end{lemma}
As mentioned in \cite{HongS2017}, lemma \ref{lemmaNR3} implies that
\begin{itemize}
\item $m(\xi)=|\xi|^{2s}+O(|\xi|^{2s-1}),\; |\xi|\rightarrow\infty$
\item $m(\xi)=s(2s-1)|A|^{2s-2}\xi^{2}+O(|\xi|^{3}),\; |\xi|\rightarrow 0$
\end{itemize}
Then $M$ behaves like $(-\partial_{xx})^{s}$ in high frequencies and like $s(2s-1)|A|^{2s-2}(-\partial_{xx})$ in low frequencies. Note also that, in order to have RE solutions in (\ref{fnls22_1}), lemma \ref{lemmaNR3} delimits the range of the fractional parameter $s$ and provides a specific value of the second Lagrange multiplier $\lambda_{0}^{2}$ in terms of $A$, which is consistent with the one obtained $A=\lambda_{0}^{2}/2$ for the classical NLS, cf. \cite{DuranS2000}.

The subfamily of solitary wave solutions (\ref{fnls22_7}) will now take the form
\begin{eqnarray}
\psi(x,t,a,c,x_{0},\theta_{0})&=&G_{(t\lambda_{0}^{1},t\lambda_{0}^{2})}(\varphi)\nonumber\\
&=&\rho(x-t\lambda_{0}^{2}-x_{0})e^{i(A(x-t\lambda_{0}^{2}-x_{0})+\theta_{0}+\lambda_{0}^{1}t)}.\label{fnls22_7b}
\end{eqnarray} 

where
\begin{itemize}
\item $a=B+|A|^{2s}=\lambda_{0}^{1}-(2s-1)|A|^{2s}$ plays the r\^{o}le of the amplitude of $\rho$.
\item $c_{s}=\lambda_{0}^{2}=2s|A|^{2s-2}A$ determines the speed.
\end{itemize}
We observe that, for the cubic case $\sigma=1$ and $s\in (1/2,1)$,  the traveling waves (\ref{fnls22_7b}) are exactly those obtained in \cite{HongS2017} with the identifications of the notation therein, cf. (\ref{fnls8a})
$$u_{0}(x)=e^{i\frac{c}{2}}Q(x),\quad \lambda_{0}^{1}=\omega^{2s}+(2s-1)\left|\frac{c}{2}\right|^{2s},\quad \lambda_{0}^{2}=2s\left|\frac{c}{2}\right|^{2s-2}\frac{c}{2}.$$ When $s=1$, the solutions (\ref{fnls22_7b}) are formulated in \cite{DuranS2000}.
\subsection{Existence of solitary waves}
\label{sec23}
In this section the existence of solutions $u_{0}=(v_{0},w_{0})$ of (\ref{fnls22_2}) will be studied using the Concentration-Compactness theory. To this end, we will consider the following family of minimization problems
\begin{eqnarray}
I_{\lambda}=\inf\{E(v,w): (v,w)\in H^{s}\times H^{s}, F(v,w)=\lambda\},\quad \lambda>0,\label{fnls_231}
\end{eqnarray}
for $s\in (1/2,1], \sigma>0$ and where
\begin{eqnarray}
E(v,w)&=&\frac{1}{2}\int_{\mathbb{R}}\left(v(-\partial_{xx})^{s}v+w(-\partial_{xx})^{s}w+\lambda_{0}^{1}(v^{2}+w^{2})\right.\nonumber\\
&&\left.-\lambda_{0}^{2}(vw_{x}-v_{x}w)\right)dx,\label{fnls_232}\\
F(v,w)&=&\frac{1}{2\sigma+2}\int_{\mathbb{R}}(v^{2}+w^{2})^{\sigma+1}dx.\nonumber
\end{eqnarray}
The functional (\ref{fnls_232}) can be written as
\begin{eqnarray}
E(v,w)=\frac{1}{2}\langle Q\begin{pmatrix}v\\w\end{pmatrix},\begin{pmatrix}v\\w\end{pmatrix}\rangle,\label{fnls_234}
\end{eqnarray}
where $Q$ is a matrix operator with Fourier symbol
\begin{eqnarray}
\widehat{Q}(\xi)=\begin{pmatrix}\lambda_{0}^{1}+|\xi|^{2s}&-i\lambda_{0}^{2}\xi\\i\lambda_{0}^{2}\xi &\lambda_{0}^{1}+|\xi|^{2s}\end{pmatrix},\quad \xi\in\mathbb{R}.\label{fnls_235}
\end{eqnarray}
\begin{lemma}
\label{lem24} Assume that $1/2<s\leq 1, \lambda_{0}^{1}>0$ and let
\begin{eqnarray}
c(\lambda_{0}^{1})=2s\left(\frac{\lambda_{0}^{1}}{2s-1}\right)^{\frac{2s-1}{2s}}.\label{fnls_236}
\end{eqnarray}
If 
\begin{eqnarray}
0<c_{s}=\lambda_{0}^{2}<c(\lambda_{0}^{1}),\label{fnls_236b}
\end{eqnarray}
 then the operator $Q$ given by (\ref{fnls_235}) is positive definite and defines a norm which is equivalent to the standard $H^{s}\times H^{s}$ norm. 
\end{lemma}
\begin{proof}
Note first that $\widehat{Q}(\xi)^{*}=\widehat{Q}(\xi), \xi\in\mathbb{R}$; then, from the representation (\ref{fnls_235}) and Plancherel identity, it holds that the operator $Q$ is Hermitian. The real eigenvalues of the matrix (\ref{fnls_235}) are
\begin{eqnarray*}
\lambda_{\pm}(\xi)=|\xi|^{2s}+\lambda_{0}^{1}\pm\lambda_{0}^{2}\xi.\label{fnls_237}
\end{eqnarray*}
(Note that $\lambda_{+}(-\xi)=\lambda_{-}(\xi)$.) We consider the functions
\begin{eqnarray*}
F_{-}(x)&=&x^{2s}+\lambda_{0}^{1}-\lambda_{0}^{2}x,\quad x\geq 0,\\
F_{+}(x)&=&|x|^{2s}+\lambda_{0}^{1}+\lambda_{0}^{2}x,\quad x\leq 0.
\end{eqnarray*}
Note that $F_{-}'(x)=0\Leftrightarrow x=x^{*}=\left(\lambda_{0}^{1}/2s\right)^{\frac{1}{2s-1}}$. Since $s>1/2$ and $F_{-}(x)\rightarrow +\infty$ as $x\rightarrow +\infty$, then $F_{-}$ attains a minimum at $x=x^{*}$ with
$$
F_{-}(x^{*})=\lambda_{0}^{1}-(2s-1)\left(\frac{\lambda_{0}^{2}}{2s}\right)^{\frac{2s}{2s-1}},
$$ which is positive under the hypothesis (\ref{fnls_236b}). This implies that
\begin{eqnarray}
\lambda_{+}(\xi)\geq \lambda_{-}(\xi)>0,\quad \xi\geq 0.\label{fnls_237b}
\end{eqnarray}
Similarly, since $F_{+}(-x)=F_{-}(x)$, then $F_{+}(x)$ attains a minimum at $x=-x^{*}$ with $F_{+}(-x^{*})>0$ by (\ref{fnls_236b}). Therefore
\begin{eqnarray}
\lambda_{-}(\xi)\geq \lambda_{+}(\xi)>0,\quad \xi\leq 0.\label{fnls_237c}
\end{eqnarray}
Thus, (\ref{fnls_237b}), (\ref{fnls_237c}) imply that $Q$ is positive definite. For the second part of the lemma, we will prove the existence of positive constants $\alpha_{j}, \beta_{j}, j=1,2$ such that for all $\xi\in\mathbb{R}$
\begin{eqnarray}
\alpha_{0}+\alpha_{1}|\xi|^{2s}<\lambda_{+}(\xi), \lambda_{-}(\xi)<\beta_{0}+\beta_{1}|\xi|^{2s}.\label{fnls_238}
\end{eqnarray}
Observe first that
$$
\lambda_{+}(\xi)+\lambda_{-}(\xi)=2(\lambda_{0}^{1}+|\xi|^{2s}).
$$ Therefore, from (\ref{fnls_237b}), (\ref{fnls_237c})
\begin{eqnarray*}
&&0<\lambda_{-}(\xi)\leq \lambda_{+}(\xi)<2(\lambda_{0}^{1}+|\xi|^{2s}),\\
&&0<\lambda_{+}(\xi)\leq \lambda_{-}(\xi)<2(\lambda_{0}^{1}+|\xi|^{2s}),\quad \xi\geq 0,
\end{eqnarray*}
and the second inequality of (\ref{fnls_238}) holds by taking $\beta_{0}=2\lambda_{0}^{1}, \beta_{1}=2$. On the other hand, the first inequality is obtained from $\alpha_{1}\in (0,1), \alpha_{0}=\alpha_{1}\lambda_{0}^{1}$, and using (\ref{fnls_236b}), (\ref{fnls_237b}), and (\ref{fnls_237c}).
\end{proof}
The following properties are required by the application of the Concentration-Compactness theory.
\begin{proposition}
\label{propos25}
Under the hypotheses of Lemma \ref{lem24}, there holds:
\begin{itemize}
\item[(i)] The functional $E:H^{s}\times H^{s}\rightarrow\mathbb{R}$ given by (\ref{fnls_232}) is well sefined and there are positive constants $C_{j}=C_{j}(\lambda_{0}^{1},\lambda_{0}^{2},s), j=1,2$ such that
\begin{eqnarray}
C_{1}||(v,w)||_{s}^{2}\leq E(v,w)\leq C_{2}||(v,w)||_{s}^{2}.\label{fnls_238b}.
\end{eqnarray}
\item[(ii)] $I_{\lambda}>0$ for $\lambda>0$.
\item[(iii)] All minimizing sequences for $I_{\lambda}, \lambda>0$, are bounded in $H^{s}\times H^{2}$.
\item[(iv)] For all $\theta\in (0,\lambda)$
\begin{eqnarray}
I_{\lambda}<I_{\theta}+I_{\lambda-\theta}.\label{fnls_238c}
\end{eqnarray}
\end{itemize}
\end{proposition}
\begin{proof}
Continuity and coercivity properties (\ref{fnls_238b}) of $E$ follow from (\ref{fnls_234}) and Lemma \ref{lem24}. On the other hand, since $s>1/2$, it holds that if $I_{\lambda}$ is attained at $(v,w)\in H^{s}\times H^{s}$, then $(v,w)\in L^{\infty}\times L^{\infty}$ and
\begin{eqnarray}
0<\lambda=F(v,w)&\leq &\frac{1}{2\sigma+2}||(v,w)||_{L^{\infty}}^{2\sigma}||(v,w)||_{L^{2}}^{2}\label{fnls_239b}\\
&\leq & \frac{1}{2\sigma+2}||(v,w)||_{s}^{2\sigma+2}.\label{fnls_239}
\end{eqnarray}
Now, using (\ref{fnls_238b}) and (\ref{fnls_239}), we have
$$I_{\lambda}=E(v,w)\geq C\lambda^{\frac{1}{\sigma+1}},$$ for some positive constant $C$, which implies (ii). Property (iii) follows from the coercivity of $E$ in (\ref{fnls_238b}). Note finally that, since $E$ and $F$ are homogeneous of degree $2$ and $2\sigma+2$ respectively, then
$$I_{\tau\lambda}=\tau^{\frac{1}{\sigma+1}}I_{\lambda}.$$ Therefore, if $\theta\in (0,\lambda)$, then we write $\theta=\tau\lambda, \tau\in (0,1)$ and since $\sigma>0$ we have
\begin{eqnarray*}
I_{\theta}+I_{\lambda-\theta}&=&I_{\tau\lambda}+I_{(1-\tau)\lambda}=\tau^{\frac{1}{\sigma+1}}I_{\lambda}+(1-\tau)^{\frac{1}{\sigma+1}}I_{\lambda}\\
&>&(\tau+(1-\tau))I_{\lambda},
\end{eqnarray*}
and (\ref{fnls_238c}) follows.
\end{proof}
The existence of solutions of (\ref{fnls22_2}) for $s\in (1/2,1], \sigma\lambda_{0}^{1}>0$ and $c_{s}=\lambda_{0}^{2}$ satisfying (\ref{fnls_236b}) will be proved via the Concentration-Compactness theory, \cite{Lions}, as well as the convergence of a minimizing sequence of (\ref{fnls_231}) modulo the symmetry group (\ref{symg}) to some $(v,w)\in H^{s}\times H^{s}$ with $\lambda=F(v,w)$. Let $\{(v_{n},w_{n})\}_{n}$ be a minimizing sequence for (\ref{fnls_231}) and consider the sequence of nonnegative functions
\begin{eqnarray*}
\rho_{n}(x)=|D^{s}v_{n}(x)|^{2}+|D^{s}w_{n}(x)|^{2}+v_{n}(x)^{2}+w_{n}(x)^{2}.\label{fnls_2310}
\end{eqnarray*}
Then $\rho_{n}\in L^{1}$ and
$\lambda_{n}=||\rho_{n}||_{L^{1}}=||(v_{n},w_{n})||_{s}^{2}.$ Then, using (\ref{fnls_238b}) and Proposition \ref{propos25}(ii), $\lambda_{n}$ is bounded and
$$\lambda_{n}>(\lambda(2\sigma+2))^{\frac{1}{\sigma+1}}.$$ 
Let $\sigma=\lim_{n\rightarrow\infty}\lambda_{n}>0$. We can normalize $\rho_{n}$ as
$\widetilde{\rho}_{n}(x)=\sigma\rho_{n}(\lambda_{n}x)$, to have
$$\widetilde{\lambda}_{n}=||\widetilde{\rho}_{n}||_{L^{1}}=\sigma.$$
(Tildes are dropped from now on.) Then we apply Lemma 1.1 of \cite{Lions} to derive the existence of a subsequence $\{\rho_{n_{k}}\}_{k\geq 1}$ satisfying one of the following three conditions:
\begin{itemize}
\item[(1)] (Compactness.) There are $y_{k}\in\mathbb{R}$ such that $\rho_{n_{k}}(\cdot+y_{k})$ satisfies that for any $\epsilon>0$ there exists $R=R(\epsilon)>0$ large enough such that
\begin{eqnarray*}
\int_{|x-y_{k}|\leq R}\rho_{n_{k}}(x)dx\geq \sigma-\epsilon.
\end{eqnarray*}
\item[(2)] (Vanishing.) For any $R>0$
\begin{eqnarray}
\lim_{k\rightarrow\infty}\sup_{y\in\mathbb{R}}\int_{|x-y|\leq R}\rho_{n_{k}}(x)dx=0.\label{fnls_2311}
\end{eqnarray}
\item[(3)] (Dichotomy.) There is $\theta_{0}\in (0,\sigma)$ such that for any $\epsilon>0$ there exists $k_{0}\geq 1$ and $\rho_{k,{1}}, \rho_{k,{2}}\in L^{1}$ with $\rho_{k_{1}}, \rho_{k_{2}}\geq 0$ such that for $k\geq k_{0}$
\begin{eqnarray}
&&\int_{\mathbb{R}}|\rho_{n_{k}}-(\rho_{k,1}+\rho_{k,2})|dx\leq \epsilon,\label{fnls_2312}\\
&&\left|\int_{\mathbb{R}}\rho_{k,1}dx-\theta_{0}\right|\leq\epsilon,\quad
\left|\int_{\mathbb{R}}\rho_{k,2}dx-(\sigma-\theta_{0})\right|\leq\epsilon,
\end{eqnarray}
with
\begin{eqnarray*}
&&{\rm supp}\rho_{k,1}\cap{\rm supp}\rho_{k,2}=\emptyset,\\
&&{\rm dist}\left({\rm supp}\rho_{k,1},{\rm supp}\rho_{k,2}\right)\rightarrow+\infty,\; k\rightarrow\infty.
\end{eqnarray*}
Since the supports of $\rho_{k,1}$ and $\rho_{k,2}$ are disjoint, we may assume the existence of $R_{0}>0$ and sequences $\{y_{k}\}_{k}$ and $R_{k}\rightarrow\infty$ as $k\rightarrow\infty$ such that, \cite{AnguloS2020}
\begin{eqnarray}
&&{\rm supp}\rho_{k,1}\subset (y_{k}-R_{0},y_{k}+R_{0}),\nonumber\\
&&{\rm supp}\rho_{k,2}\subset (-\infty,y_{k}-2R_{k})\cup (y_{k}+2R_{k},\infty).\label{fnls_2312a}
\end{eqnarray}
Thus, using (\ref{fnls_2312}), (\ref{fnls_2312a}) it holds that, \cite{AnguloS2020}
\begin{eqnarray}
&&\int_{|x-y_{k}|\leq R_{0}}|\rho_{n_{k}}(x)-\rho_{k,1}(x)|dx\leq \epsilon\label{fnls_2312b}\\
&&\int_{|x-y_{k}|\geq 2R_{k}}|\rho_{n_{k}}(x)-\rho_{k,2}(x)|dx\leq \epsilon,\;
\int_{R_{0}\leq |x-y_{k}|\leq 2R_{k}}\rho_{n_{k}}(x)dx\leq \epsilon.\nonumber
\end{eqnarray}
\end{itemize}
We first rule out the vanishing property. If (\ref{fnls_2311}) hods, then
\begin{eqnarray*}
\lim_{k\rightarrow\infty}\sup_{y\in\mathbb{R}}\int_{|x-y|\leq R}\left(|v_{n_{k}}(x)|^{2}+|w_{n_{k}}(x)|^{2}\right)dx=0.
\end{eqnarray*}
Since $\{(v_{n},w_{n})\}_{n}$ is a bounded sequence in $H^{s}\times H^{s}$ (cf. Proposition \ref{propos25}(iii)), we apply (\ref{fnls_239b}) to have
\begin{eqnarray*}
F(v_{n_{k}},w_{n_{k}})\leq C||(v_{n_{k}},w_{n_{k}})||_{L^{2}}^{2},
\end{eqnarray*}
for some constant $C$. Therefore
$$\lim_{k\rightarrow\infty}F(v_{n_{k}},w_{n_{k}})=0,$$ which contradicts the fact that $\lambda>0$.

We now assume that dichotomy holds and use similar arguments to those of \cite{AnguloS2020} to rule it out. Let $h_{n_{k}}=(v_{n_{k}},w_{n_{k}})$ and cutoff functions $\varphi,\phi\in C^{\infty}(\mathbb{R}), 0\leq\varphi,\phi\leq 1$, with
\begin{eqnarray*}
&&\phi(x)=1,\quad |x|\leq 1,\quad \phi(x)=0,\quad |x|\geq 2,\\
&&\varphi(x)=1,\quad |x|\geq 2,\quad \varphi(x)=0,\quad |x|\leq 1.
\end{eqnarray*}
Let $R_{1}>R_{0}$ and
\begin{eqnarray*}
\phi_{k}=\phi\left(\frac{x-y_{k}}{R_{1}}\right),\quad 
\varphi_{k}=\varphi\left(\frac{x-y_{k}}{R_{k}}\right),\quad x\in\mathbb{R}.
\end{eqnarray*}
Then
\begin{eqnarray*}
&&{\rm supp}\phi_{k}\subset (y_{k}-2R_{1},y_{k}+2R_{1}),\\
&&{\rm supp}\varphi_{k}\subset (-\infty,y_{k}-2R_{k})\cup (y_{k}+2R_{k},\infty).
\end{eqnarray*}
We define $h_{k,1}=\phi_{k}h_{n_{k}}, h_{k,2}=\varphi_{k}h_{n_{k}}$, and 
$$z_{k}=h_{n_{k}}-(h_{k,1}+h_{k,2})=\chi_{k} h_{n_{k}},\quad \chi_{k}=1-\phi_{k}-\varphi_{k}.$$
Since $\chi_{k}\in C^{\infty}(\mathbb{R})$ and ${\rm supp}\chi_{k}\subset \{x\in\mathbb{R}/ R_{1}\leq |x-y_{k}|\leq R_{k}\}$, then its Fourier transform decays exponentially. Using this fact, (\ref{fnls_2312b}), and the Kato-Ponce inequalities, \cite{KP,KPV}
\begin{eqnarray*}
||D^{s}(\chi_{k}g)||_{L^{2}}\leq C\left(||D^{s}\chi_{k}||_{L^{\infty}}||g||_{L^{2}}+||\chi_{k}||_{L^{\infty}}||D^{s}g||_{L^{2}}\right),
\end{eqnarray*}
with $g=v_{n_{k}},w_{n_{k}}$, we have
\begin{eqnarray*}
||\chi_{k}v_{n_{k}}||_{s}^{2}\leq C\int_{R_{1}\leq |x-y_{k}|\leq R_{k}}\rho_{n_{k}}(x)dx=O(\epsilon),\quad ||\chi_{k}w_{n_{k}}||_{s}^{2}=O(\epsilon).
\end{eqnarray*}
Therefore
\begin{eqnarray}
||z_{k}||_{s}=O(\epsilon).\label{fnls_2313}
\end{eqnarray}
On the other hand, since
\begin{eqnarray*}
F(\phi_{k}v_{n_{k}},\phi_{k}w_{n_{k}})=\frac{1}{2\sigma+2}\int_{\mathbb{R}}\phi_{k}^{2\sigma+2}(v_{n_{k}}^{2}+w_{n_{k}}^{2})^{\sigma+1}dx,
\end{eqnarray*}
is bounded, then there is a subsequence of $h_{k,1}$ (denoted again by $h_{k,1}$) and $\theta\in\mathbb{R}$ such that
\begin{eqnarray}
\left|\frac{1}{2\sigma+2}\int_{\mathbb{R}}\phi_{k}^{2\sigma+2}(v_{n_{k}}^{2}+w_{n_{k}}^{2})^{\sigma+1}dx-\theta\right|\leq \epsilon,\quad k\geq k_{0}.\label{fnls_2313b}
\end{eqnarray}
Furthermore, since $\lambda=F(v_{n_{k}},w_{n_{k}})$
\begin{eqnarray}
&&\left|\frac{1}{2\sigma+2}\int_{\mathbb{R}}\varphi_{k}^{2\sigma+2}(v_{n_{k}}^{2}+w_{n_{k}}^{2})^{\sigma+1}dx-(\lambda-\theta)\right|\nonumber\\
&&=\left|\frac{1}{2\sigma+2}\int_{\mathbb{R}}(\varphi_{k}^{2\sigma+2}-1)(v_{n_{k}}^{2}+w_{n_{k}}^{2})^{\sigma+1}dx+\theta\right|\nonumber\\
&&\leq \left|\frac{1}{2\sigma+2}\int_{\mathbb{R}}(\varphi_{k}^{2\sigma+2}+\phi_{k}^{2\sigma+2}-1)(v_{n_{k}}^{2}+w_{n_{k}}^{2})^{\sigma+1}dx\right|+\epsilon\label{fnls_2314}\\
&&\leq  \frac{1}{2\sigma+2}\int_{\mathbb{R}}\chi_{k}^{2\sigma+2}(v_{n_{k}}^{2}+w_{n_{k}}^{2})^{\sigma+1}dx+\epsilon\leq C||z_{k}||_{s}^{2\sigma+2}+\epsilon=O(\epsilon).\nonumber
\end{eqnarray}
The next step consists of proving that for $R_{1}$ and $R_{k}$ large enough
\begin{eqnarray}
E(h_{n_{k}})=E(h_{k,1})+E(h_{k,2})+O(\epsilon).\label{fnls_2315}
\end{eqnarray}
Note first that using (\ref{fnls_234}) and Lemma \ref{lem24} we can write
\begin{eqnarray}
E(h_{n_{k}})&=&E(h_{k,1})+E(h_{k,2})+E(z_{k})\nonumber\\
&&+\langle Qz_{k},h_{k,21}+h_{k,2}\rangle+\langle Qh_{k,1},h_{k,2}\rangle.\label{fnls_2316}
\end{eqnarray}
From Lemma \ref{lem24} and (\ref{fnls_2313}), we have
\begin{eqnarray}
\left|\langle Qz_{k},h_{k,21}+h_{k,2}\rangle\right|&\leq &||Qz_{k}||_{L^{2}}||h_{k,1}+h_{k,2}||_{L^{2}}\nonumber\\
&\leq & C||z_{k}||_{s}\left(||h_{k,1}||_{L^{2}}+||h_{k,2}||_{L^{2}}\right)\nonumber\\
&\leq &C||z_{k}||_{s}||h_{n_{k}}||_{L^{2}}=O(\epsilon).\label{fnls_2316b}
\end{eqnarray}
For the control of the last term of (\ref{fnls_2316}), we write
\begin{eqnarray}
\langle Qh_{k,1},h_{k,2}\rangle=\int_{\mathbb{R}}\left(\varphi_{k}v_{n_{k}}D^{2s}(\phi_{k}v_{n_{k}})+\varphi_{k}w_{n_{k}}D^{2s}(\phi_{k}w_{n_{k}})\right)dx.\label{fnls_2317}
\end{eqnarray}
Using the definition of the Calderon-commutator (\ref{Calc}) and the fact that $\phi_{k}\varphi_{k}=0$, we have
\begin{eqnarray*}
\varphi_{k}v_{n_{k}}D^{2s}(\phi_{k}v_{n_{k}})&=&\varphi_{k}v_{n_{k}}\left(\phi_{k}D^{2s}v_{n_{k}}+[D^{2s},\phi_{k}]v_{n_{k}}\right)\\
&=&\varphi_{k}v_{n_{k}}[D^{2s},\phi_{k}]v_{n_{k}}.
\end{eqnarray*}
Thus, Kato-Ponce inequality, \cite{KP,KPV}, leads to
\begin{eqnarray}
\left\|[D^{2s},\phi_{k}]v_{n_{k}}\right\|_{L^{2}}\leq C\left(||\varphi_{k}'||_{L^{\infty}}||D^{2s-1}v_{n_{k}}||_{L^{2}}+||D^{2s}\phi_{k}||_{L^{\infty}}||v_{n_{k}}||_{L^{2}}\right).\label{fnls_2318}
\end{eqnarray}
Since $s\in (1/2,1]$ then $2s>1$ and $s\geq 2s-1$; the embeddings $H^{2s}\subset H^{1}, H^{s}\subset H^{2s-1}$ imply that the right-hand side of (\ref{fnls_2318}) is bounded. Therefore
\begin{eqnarray*}
\left|\int_{\mathbb{R}}\varphi_{k}v_{n_{k}}D^{2s}(\phi_{k}v_{n_{k}})dx\right|&=&
\left|\int_{\mathbb{R}}\varphi_{k}v_{n_{k}}[D^{2s},\phi_{k}]v_{n_{k}}dx\right|\\
&\leq & ||\varphi_{k}||_{L^{\infty}}||v_{n_{k}}||_{L^{2}}C\left\|[D^{2s},\phi_{k}]v_{n_{k}}\right\|_{L^{2}}\\
&\leq & C\frac{||\phi'||_{L^{\infty}}}{R_{1}}||v_{n_{k}}||_{s}^{2},
\end{eqnarray*}
and similarly
\begin{eqnarray*}
\left|\int_{\mathbb{R}}\varphi_{k}w_{n_{k}}D^{2s}(\phi_{k}w_{n_{k}})dx\right|\leq C\frac{||\phi'||_{L^{\infty}}}{R_{1}}||w_{n_{k}}||_{s}^{2}.
\end{eqnarray*}
Thus for $R_{1}$ large enough we obtain that (\ref{fnls_2317}) is $O(\epsilon)$. Using this fact, (\ref{fnls_2316}), (\ref{fnls_2316b}), and Proposition \ref{propos25}(i), then (\ref{fnls_2315}) follows. Moreover
\begin{eqnarray}
I_{\lambda}&\geq & \lim_{k\rightarrow\infty}\inf E(h_{n_{k}})\geq  \lim_{k\rightarrow\infty}\inf E(h_{k,1})\nonumber\\
&&+\lim_{k\rightarrow\infty}\inf E(h_{k,2})+O(\epsilon).\label{fnls_2318b}
\end{eqnarray}

The last estimates that we need are the following: For $R_{1},R_{k}$ large enough
\begin{eqnarray}
\left|||\phi_{k}v_{n_{k}}||_{s}^{2}+||\phi_{k}w_{n_{k}}||_{s}^{2}-\int_{\mathbb{R}}\rho_{k,1}dx\right|&=&O(\epsilon),\label{fnls_2319}\\
\left|||\varphi_{k}v_{n_{k}}||_{s}^{2}+||\varphi_{k}w_{n_{k}}||_{s}^{2}-\int_{\mathbb{R}}\rho_{k,2}dx\right|&=&O(\epsilon).\label{fnls_2320}
\end{eqnarray}
For the proof, we note that using the Calderon commutator (\ref{Calc}), the left-hand side of (\ref{fnls_2319}) can be written as
\begin{eqnarray}
&&\left|||\phi_{k}v_{n_{k}}||_{s}^{2}+||\phi_{k}w_{n_{k}}||_{s}^{2}-\int_{\mathbb{R}}\rho_{k,1}dx\right|\nonumber\\
&&\left|\int_{\mathbb{R}}\left(\phi_{k}^{2}(v_{n_{k}}^{2}+w_{n_{k}}^{2})+\phi_{k}v_{n_{k}}\left(\phi_{k}D^{2s}v_{n_{k}}+[D^{2s},\phi_{k}]v_{n_{k}}\right)\right.\right.\nonumber\\
&&\left.\left.\phi_{k}w_{n_{k}}\left(\phi_{k}D^{2s}w_{n_{k}}+[D^{2s},\phi_{k}]w_{n_{k}}\right)-\rho_{k,1}\right)dx\right|.\label{fnls_2319c}
\end{eqnarray}
Then, (\ref{fnls_2319}) follows from the property $0\leq \phi_{k}\leq 1$, (\ref{fnls_2312}), and (\ref{fnls_2318}). For the proof of (\ref{fnls_2320}), we can write its left-hand side in the same way as in (\ref{fnls_2319c}) and apply that $0\leq \varphi_{k}\leq 1$, (\ref{fnls_2312}), and the arguments of (\ref{fnls_2318}) to have
\begin{eqnarray*}
\left|\int_{\mathbb{R}}\varphi_{k}v_{n_{k}}D^{2s}(\varphi_{k}v_{n_{k}})dx\right|&\leq & C\frac{||\varphi'||_{L^{\infty}}}{R_{k}}||v_{n_{k}}||_{s}^{2},\\
\left|\int_{\mathbb{R}}\varphi_{k}w_{n_{k}}D^{2s}(\varphi_{k}w_{n_{k}})dx\right|&\leq & C\frac{||\varphi'||_{L^{\infty}}}{R_{k}}||w_{n_{k}}||_{s}^{2}.
\end{eqnarray*}

We are now ready to rule out dichotomy. Concerning the limit $\theta$ in (\ref{fnls_2313b}), we have the following possibilities:
\begin{itemize}
\item[(1)] $\theta=0$. Then, from (\ref{fnls_2314}) and $k$ large
\begin{eqnarray*}
|F(h_{k,2})-\lambda|<\epsilon.
\end{eqnarray*}
If we tale $\epsilon<\lambda/2$ then
\begin{eqnarray}
F(h_{k,2})>\lambda-\epsilon>\frac{\lambda}{2}>0.\label{fnls_2320b}
\end{eqnarray}
Thus, if we consider 
$$d_{k}=\left(\frac{\lambda}{F(h_{k,2})}\right)^{\frac{1}{2\sigma+2}},$$ then $\lambda=F(d_{k}h_{k,2})$ and
\begin{eqnarray}
|d_{k}-1|&<&\left(\frac{2}{\lambda}\right)^{\frac{1}{2\sigma+2}}\left|\lambda^{\frac{1}{2\sigma+2}}-F(h_{k,2})^{\frac{1}{2\sigma+2}}\right|\nonumber\\
&<&\left(\frac{2}{\lambda}\right)^{\frac{1}{2\sigma+2}}\frac{\left|\lambda^{\frac{1}{2\sigma+2}}-F(h_{k,2})^{\frac{1}{2\sigma+2}}\right|}{\left|\lambda-F(h_{k,2})\right|}\epsilon.\label{fnls_2321}
\end{eqnarray}
For $x>0$ let $g(x)=x^{\frac{1}{2\sigma+2}}$. The numerator in (\ref{fnls_2321}) can be written as
\begin{eqnarray*}
\left|\lambda^{\frac{1}{2\sigma+2}}-F(h_{k,2})^{\frac{1}{2\sigma+2}}\right|&=&\left|g(\lambda)-g(F(h_{k,2}))\right|\\
&=&\left|g'(\xi_{k})(\lambda-F(h_{k,2}))\right|,
\end{eqnarray*}
for some $\xi_{k}$ between $F(h_{k,2})$ and $\lambda$ and consequently, by (\ref{fnls_2320b}), between $\lambda/2$ and $\lambda$. Therefore
\begin{eqnarray}
g'(\xi_{k})=\frac{1}{2\sigma+2}\left(\frac{1}{\xi_{k}}\right)^{\frac{2\sigma+1}{2\sigma+2}}<\frac{1}{2\sigma+2}\left(\frac{2}{\lambda}\right)^{\frac{2\sigma+1}{2\sigma+2}}.\label{fnls_2322}
\end{eqnarray}
Inserting (\ref{fnls_2322}) into (\ref{fnls_2321}) yields
\begin{eqnarray*}
|d_{k}-1|<\frac{1}{2\sigma+2}\left(\frac{2}{\lambda}\right)\epsilon.
\end{eqnarray*}
Thus $\lim_{k\rightarrow\infty}d_{k}=1$ and
\begin{eqnarray}
I_{\lambda}\leq E(d_{k}h_{k,2})=d_{k}^{2}E(h_{k,2})=E(h_{k,2})+O(\epsilon).\label{fnls_2323}
\end{eqnarray}
Now, from Proposition \ref{propos25}(i), (\ref{fnls_2319}), and (\ref{fnls_2312}), we have 
\begin{eqnarray*}
\lim_{k\rightarrow\infty}\inf E(h_{k,1})&\geq &C\lim_{k\rightarrow\infty}\inf ||h_{k,1}||_{s}^{2}\\
&\geq & C\lim_{k\rightarrow\infty}\inf ||\rho_{k,1}||_{L^{1}}+O(\epsilon)\geq C\theta_{0}+O(\epsilon).
\end{eqnarray*}
Therefore, (\ref{fnls_2318b}) and (\ref{fnls_2323}) imply, for $\epsilon>0$ arbitrarily small
\begin{eqnarray*}
I_{\lambda}\geq C\theta_{0}+I_{\lambda}+O(\epsilon),
\end{eqnarray*}
which leads to $I_{\lambda}\geq C\theta_{0}+I_{\lambda}$, that is, the contradiction $\theta_{0}\leq 0$.
\item[(2)] If we assume $\lambda>\theta>0$, then (\ref{fnls_2313b}), (\ref{fnls_2314}), and the previous arguments applied to $\lambda-\theta$ and $\theta$ lead to
\begin{eqnarray*}
I_{\lambda-\theta}\leq E(h_{k,2})+O(\epsilon),\quad 
I_{\theta}\leq E(h_{k,1})+O(\epsilon).
\end{eqnarray*}
Therefore, from (\ref{fnls_2318b}), it holds that
\begin{eqnarray*}
I_{\lambda}\geq I_{\lambda-\theta}+I_{\theta}+O(\epsilon),
\end{eqnarray*}
for $\epsilon>0$ arbitrarily small. This contradicts Proposition \ref{propos25}(iv).
\item[(3)] If $\theta<0$ then, from (\ref{fnls_2314})
\begin{eqnarray*}
\lim_{k\rightarrow\infty}F(h_{k,2})=\lambda-\theta>\frac{\lambda}{2},
\end{eqnarray*}
and a similar analysis to item (1) leads to contradiction.
\item[(4)] If $\theta=\lambda$, then 
$$\lim_{k\rightarrow\infty}F(h_{k,1})=\lambda;$$ therefore, as in item (1) we may take $\epsilon$ small enough so that $F(h_{k,1})>\lambda/2$ for $k$ large. The same arguments apply to have
\begin{eqnarray*}
I_{\lambda}\leq E(h_{k,1})+O(\epsilon),
\end{eqnarray*}
and
\begin{eqnarray*}
\lim_{k\rightarrow\infty}\inf E(h_{k,2})&\geq &C\lim_{k\rightarrow\infty}\inf ||h_{k,2}||_{s}^{2}\\
&\geq & C\lim_{k\rightarrow\infty}\inf ||\rho_{k,2}||_{L^{1}}+O(\epsilon)\geq C(\sigma-\theta_{0})+O(\epsilon),
\end{eqnarray*}
leading to, for $\epsilon>0$ arbitrarily small
\begin{eqnarray*}
I_{\lambda}\geq C(\sigma-\theta_{0})+I_{\lambda}+O(\epsilon),
\end{eqnarray*}
implying the contradiction $\sigma-\theta_{0}\leq 0$. 
\item[(5)] If $\theta>\lambda$, then $F(h_{k,1})>0$ for $k$ large enough. Let
\begin{eqnarray*}
e_{k}=\left(\frac{\lambda}{F(h_{k,1})}\right)^{\frac{1}{2\sigma+2}}.
\end{eqnarray*}
Then $F(e_{k}h_{k,1})=\lambda$ and
$$\lim_{k\rightarrow\infty}e_{k}=\left(\frac{\lambda}{\theta}\right)^{\frac{1}{2\sigma+2}}.$$ Therefore, for $k$ large enough, we have
\begin{eqnarray}
I_{\lambda}\leq E(e_{k}h_{k,1})=e_{k}^{2}E(h_{k,1})<E(h_{k,1}),\label{fnls_2324a}
\end{eqnarray}
and, on the other hand, from Proposition \ref{propos25}(i), (\ref{fnls_2320}), and (\ref{fnls_2321}) there holds
\begin{eqnarray}
\lim_{k\rightarrow\infty}\inf E(h_{k,2})&\geq &C\lim_{k\rightarrow\infty}\inf ||h_{k,2}||_{s}^{2}
\geq  C\lim_{k\rightarrow\infty}\inf ||\rho_{k,2}||_{L^{1}}+O(\epsilon)\nonumber\\
&\geq &C(\sigma-\theta_{0})+O(\epsilon).\label{fnls_2324b}
\end{eqnarray}
Thus, (\ref{fnls_2318b}), (\ref{fnls_2324a}), and (\ref{fnls_2324b}) imply
\begin{eqnarray*}
I_{\lambda}\geq I_{\lambda}+C(\sigma-\theta_{0})+O(\epsilon),
\end{eqnarray*}
for $\epsilon$ arbitrarily small, leading again to the contradiction $\sigma-\theta_{0}\leq 0$.
\end{itemize}
Hence, dichotomy is ruled out. Since vanishing and dichotomy do not hold, then Lemma 1.1 of \cite{Lions} implies that necessarily compactness is satisfied. Let $\epsilon>0$ and
$$P_{k}=\{x\in\mathbb{R}/ |x-y_{k}|\geq R(\epsilon)\}.$$ Using (\ref{fnls_239b}) we have
\begin{eqnarray*}
\frac{1}{2\sigma+2}\int_{P_{k}}\left(v_{n_{k}}^{2}+w_{n_{k}}^{2}\right)^{\sigma+1}dx&\leq &C||(v_{n_{k}},w_{n_{k}})||_{s}^{2\sigma}\left(\int_{P_{k}}\rho_{n_{k}}(x)dx\right)\\
&=&O(\epsilon).
\end{eqnarray*}
Therefore
\begin{eqnarray}
\left|\frac{1}{2\sigma+2}\int_{|x-y_{k}|\leq R}\left(v_{n_{k}}^{2}+w_{n_{k}}^{2}\right)^{\sigma+1}dx-\lambda\right|\leq \epsilon.\label{fnls_2325}
\end{eqnarray}
Let $\widetilde{h}_{n_{k}}=(\widetilde{v}_{n_{k}},\widetilde{w}_{n_{k}})$ with
$$\widetilde{v}_{n_{k}}(x)=v_{n_{k}}(x-y_{k}),\quad
\widetilde{w}_{n_{k}}(x)=w_{n_{k}}(x-y_{k}).$$ Then $\widetilde{h}_{n_{k}}$ is bounded in $H^{s}\times H^{s}$, so there is a subsequence (denoted again by $\widetilde{h}_{n_{k}}$) which converges weakly in $H^{s}\times H^{s}$ to some $(\widetilde{v}_{0},\widetilde{w}_{0})\in H^{s}\times H^{s}$. Since (\ref{fnls_2325}) implies that
\begin{eqnarray*}
\lambda\geq \int_{-R}^{R}\frac{1}{2\sigma+2}\left(\widetilde{v}_{n_{k}}^{2}+\widetilde{w}_{n_{k}}^{2}\right)^{\sigma+1}dx\geq \lambda-\epsilon,
\end{eqnarray*}
for $k$ large enough then, from the compact embedding, \cite{Adams1975}
$$H^{s}(-R,R)\subset L^{\sigma+1}(-R,R), \sigma>0,$$ we have
\begin{eqnarray*}
\lambda\geq \int_{-R}^{R}\frac{1}{2\sigma+2}\left(\widetilde{v}_{0}^{2}+\widetilde{w}_{0}^{2}\right)^{\sigma+1}dx\geq \lambda-\epsilon.
\end{eqnarray*}
As $\epsilon\rightarrow 0$ then $R\rightarrow\infty$, \cite{AnguloS2020}, leading to $\lambda=F(\widetilde{v}_{0},\widetilde{w}_{0})$. In addition, weak lower semicontinuity and invariance under translations of $E$ imply
\begin{eqnarray*}
I_{\lambda}&=&\lim_{k\rightarrow\infty}\inf E(v_{n_{k}},w_{n_{k}})=\lim_{k\rightarrow\infty}\inf E(\widetilde{v}_{n_{k}},\widetilde{w}_{n_{k}})\\
&\geq &E(\widetilde{v}_{0},\widetilde{w}_{0})\geq I_{\lambda}.
\end{eqnarray*}
Therefore $\widetilde{h}=(\widetilde{v}_{0},\widetilde{w}_{0})$ is a solution of the variational problem
\begin{eqnarray*}
\delta E(\widetilde{h})=K\delta F(\widetilde{h}),\quad \lambda=F(\widetilde{h}),
\end{eqnarray*}
for some Lagrange multiplier $K$. This leads to
\begin{eqnarray}
(-\partial_{xx})^{s}\widetilde{v}_{0}+\lambda_{0}^{1}\widetilde{v}_{0}-\lambda_{0}^{2}\widetilde{w}_{0}'&=&K(\widetilde{v}_{0}^{2}+\widetilde{w}_{0}^{2})^{\sigma}\widetilde{v}_{0},\nonumber\\
(-\partial_{xx})^{s}\widetilde{w}_{0}+\lambda_{0}^{1}\widetilde{w}_{0}+\lambda_{0}^{2}\widetilde{v}_{0}'&=&K(\widetilde{v}_{0}^{2}+\widetilde{w}_{0}^{2})^{\sigma}\widetilde{w}_{0}.\label{fnls_2326}
\end{eqnarray}
Multiplying the first equation of (\ref{fnls_2326}) by $\widetilde{v}_{0}$, the second by $\widetilde{w}_{0}$, adding nthe resulting equations and integrating on $\mathbb{R}$ lead to
\begin{eqnarray*}
2E(\widetilde{h})=2(\sigma+1)K F(\widetilde{h}).
\end{eqnarray*}
Therefore $K=I_{\lambda}/\lambda(\sigma+1)>0$. Defining
$$(v_{0},w_{0})=K^{\alpha}(\widetilde{v}_{0},\widetilde{w}_{0}),\quad \alpha=\frac{1}{2\sigma},$$ then it holds that $(v_{0},w_{0})$ is a solution of (\ref{fnls22_2}). This completes the proof of the first part of the following theorem.

\begin{theorem}
\label{theor26} Under the conditions of Lemma \ref{lem24}, there is a solution $(v_{0},w_{0})\in H^{s}\times H^{s}$ of (\ref{fnls22_2}) which satisfies $v_{0}, w_{0}\in H^{\infty}$.
\end{theorem}
\begin{proof}
It only remains to prove the last statement. Note that (\ref{fnls22_2}) can be written as
\begin{eqnarray}
Q\begin{pmatrix}v\\w\end{pmatrix}=G(v,w)=(v^{2}+w^{2})^{\sigma+1}\begin{pmatrix}v\\w\end{pmatrix},\label{fnls_2327}
\end{eqnarray}
where $Q$ is defined in (\ref{fnls_235}). A similar argument to that used to prove (\ref{fnls_239}) implies that $G(v,w)\in L^{2}\times L^{2}$. Then, using Lemma \ref{lem24}, we have $(v,w)\in H^{2s}\times H^{2s}$; then the result follows from a bootstrap argument applied to (\ref{fnls_2327}).
\end{proof}
\begin{remark}
\label{rema231}
Note that, for fixed $\lambda_{0}^{1}>0$, if $(\lambda_{0}^{2},v_{0},w_{0})$ is a solution of (\ref{fnls22_2}), then $(-\lambda_{0}^{2},w_{0},v_{0})$ is a solution. This implies that Theorem \ref{theor26} is valid if in Lemma \ref{lem24} we assume $0<|c_{s}|<c(\lambda_{0}^{1})$ instead of (\ref{fnls_236b}) 
\end{remark}

\subsubsection{The particular case (\ref{fnls22_7b})}
\label{sec231}
The case of the subfamily (\ref{fnls22_7b}), where $\theta(x)=Ax$ and $\rho$ satisfies (\ref{fnls22_5}), deserves some comments. Indeed, the existence of solution is ensured by Theorem \ref{theor26}. But an alternative for the proof can be given from (\ref{fnls22_5}), Lemma \ref{lemmaNR3}, and the minimization problem
\begin{eqnarray}
\theta_{\lambda}=\inf\{J(u): u\in H^{s}(\mathbb{R}): \int_{\mathbb{R}}u^{2\sigma+2}dx=\lambda\},\quad \lambda>0,\label{fnls_258c1}
\end{eqnarray} 
where
\begin{eqnarray}
J(u)=\int_{\mathbb{R}}u(M+a)udx.\label{fnls_258c2}
\end{eqnarray}
Note that the symbol (\ref{fnls22_4a}) can be written as $r(\xi)=l(\xi+A)$, where
\begin{eqnarray}
l(x)=|x|^{2s}-\lambda_{0}^{2}x+\lambda_{0}^{1},\quad x\in\mathbb{R}.\label{fnls_258b}
\end{eqnarray}
A similar study to that made in Lemma \ref{lem24} leads to the existence of positive constants $\alpha_{j},\beta_{j}, j=0,1$, such that
\begin{eqnarray*}
\alpha_{0}+\alpha_{1}|x|^{2s}\leq l(x)\leq \beta_{0}+\beta_{1}|x|^{2s},\quad x\in \mathbb{R}.
\end{eqnarray*}
Then, the existence of $\rho$ can be alternatively derived from \cite{Weinstein1987}. On the other hand, (\ref{fnls_258b}) and Proposition 1.1(iii) of \cite{FrankL} determine the asymptotic decay of the waves:
\begin{theorem}
\label{theor27} Assume that $s\in (1/2,1)$, $\lambda_{0}^{1}>0$, $c(\lambda_{0}^{1})$ is as in (\ref{fnls_236}) and $0<|c_{s}|<c(\lambda_{0}^{1})$. Then there is a solution $\rho$ of (\ref{fnls22_5}). Furthermore, there exists a constant $C>0$ such that
\begin{eqnarray}
|x|^{2s+1}\rho(x)\leq C,\quad x\in\mathbb{R}.\label{fnls_2328}
\end{eqnarray}
\end{theorem}
In addition, we observe that (\ref{fnls_258c1}), (\ref{fnls_258c2}) is equivalent, by rescaling, to the problem
\begin{eqnarray}
m=\min\{\Lambda(f):f\in H^{s}(\mathbb{R}), f\neq 0\},\label{fnls_258d1}
\end{eqnarray}
where
\begin{eqnarray}
\Lambda(f)=\frac{J(f)}{\left(\int_{\mathbb{R}}f^{2\sigma+2}dx\right)^{\frac{1}{\sigma+1}}},\label{fnls_258d2}
\end{eqnarray}
in the following sense, \cite{ChenB}: Note first that if $f$ is a minimizer for (\ref{fnls_258d1}), (\ref{fnls_258d2}), then it is also a minimizer for (\ref{fnls_258c1}), (\ref{fnls_258c2}) with
$$\theta_{\lambda}=\lambda^{\frac{1}{\sigma+1}}m,\quad \lambda=\int_{\mathbb{R}}f^{2\sigma+2}dx.$$ On the other hand, it holds that
$\Lambda'(f)h=0$ for any $h\in H^{s}(\mathbb{R})$. After some computations, this implies
\begin{eqnarray*}
(M+a)f=\frac{J(f)}{\int_{\mathbb{R}}f^{2\sigma+2}dx}f^{2\sigma+1}.
\end{eqnarray*}
Therefore, the rescaling $\rho=Cf$ with
$$C=\frac{m^{\frac{1}{2\sigma}}}{\lambda^{\frac{1}{2\sigma+2}}},$$ is a solution of (\ref{fnls22_5}). The equivalence can be used to ensure the existence of an even solution of (\ref{fnls22_5}) from the arguments shown in \cite{Albert,ChenB}, by searching for a minimizer of (\ref{fnls_258c1}), (\ref{fnls_258c2}) with positive Fourier transform.


\section{Numerical generation of solitary waves}\label{sec3}
In this section some additional properties of the solitary waves will be studied by computational means. To this end, the numerical procedure to compute approximate solitary wave profiles will be briefly described and its performance will be checked in a first group of numerical experiments. The accuracy of the method will give us some confidence to study computationally some properties of the waves concerning, among others, even character, asymptotic decay, and the speed-amplitude relation.
\subsection{Approximation to the profiles}
\label{sec31}
Contrary to the classical case (\ref{fnls22_7b}) with $s=1$, \cite{DuranS2000}, explicit formulas for the solitary-wave profiles $(v_{0},w_{0})$, solutions of (\ref{fnls22_2}), or solutions $\rho$ of (\ref{fnls22_5}), are not known in general. In order to study additional properties and dynamics of the resulting solitary waves, a numerical procedure to compute approximations is here introduced. The approach consists of solving iteratively (\ref{fnls22_2}), written in the form (\ref{fnls_2327}). This contains a nonsingular, linear part (given by the operator $Q$) and a nonlinear term $G$, homogeneous of degree $2\sigma+1$. This homogeneous character prevents the use of the classical fixed-point iteration
\begin{eqnarray}
Qz^{[\nu+1]}=G(z^{[\nu]}),\quad \nu=0,1,\ldots,\label{fnls_310}
\end{eqnarray}
where $z^{[\nu]}=(v^{[\nu]},w^{[\nu]}), \nu=0,1,\ldots,$ and for some initial iteration $z^{[0]}$. The reason is that, from Euler's Homogeneous Function Theorem
$$G'(v,w)\begin{pmatrix}v\\w\end{pmatrix}=(2\sigma+1)G(v,w),$$ which implies, if $(v_{0},w_{0})$ is a solution of (\ref{fnls22_2}), that
\begin{eqnarray*}
Q^{-1}G'(v_{0},w_{0})\begin{pmatrix}v_{0}\\w_{0}\end{pmatrix}=(2\sigma+1)Q^{-1}G(v_{0},w_{0})=(2\sigma+1)\begin{pmatrix}v_{0}\\w_{0}\end{pmatrix}.
\end{eqnarray*}
Therefore, $\lambda=2\sigma+1>1$ becomes one of the eigenvalues of the corresponding iteration operator of (\ref{fnls_310}) at the solitary-wave profile $(v_{0},w_{0})$. 

An efficient alternative is given by the Petviashvili method, \cite{Petv1976}. This is formulated as
\begin{eqnarray}
m_{\nu}&=&\frac{\langle Qz^{[\nu]},z^{[\nu]}\rangle}{\langle G(z^{[\nu]}),z^{[\nu]}\rangle},\nonumber\\
Qz^{[\nu+1]}&=&m_{\nu}^{\alpha}G(z^{[\nu]}),\quad \nu=0,1,\ldots,\label{fnls_311}
\end{eqnarray}
where $\langle\cdot,\cdot\rangle$ is given by (\ref{fnls_234b}) and $\alpha\in (1,(2\sigma+2)/2\sigma)$, with $\alpha=(2\sigma+1)/2\sigma$ as optimal choice, \cite{pelinovskys}. The iteration (\ref{fnls_311}) is characterized by the so-called stabilizing factor $m_{\nu}$. Compared to the classical fixed-point iteration, this factor modifies the spectrum of the resulting iteration operator in such a way that, from (\ref{fnls_310}) to (\ref{fnls_311}), the \lq harmful\rq\ eigenvalue $\lambda=2\sigma+1$ is transformed to some below one in magnitude for (\ref{fnls_311}) (zero in the case of the optimal choice of $\alpha$), and the rest of the spectrum does not change, \cite{AlvarezD2014}. Not being the purpose of the present paper, it is worth mentioning that this property seems to enable for obtaining convergence results for (\ref{fnls_311}) in a local, orbital sense, \cite{AlvarezD2014}. For the particular case $\theta(x)=Ax$, the Petviashvili method is applied to obtain approximations to the solutions $\rho$ of (\ref{fnls22_5}). In this case, the formulation is
\begin{eqnarray}
m_{\nu}&=&\frac{( (M+a)\rho^{[\nu]},\rho^{[\nu]})}{((\rho^{[\nu]})^{2\sigma+1},\rho^{[\nu]})},\nonumber\\
(M+a)\rho^{[\nu+1]}&=&m_{\nu}^{\alpha}(\rho^{[\nu]})^{2\sigma+1},\quad \nu=0,1,\ldots,\label{fnls_312}
\end{eqnarray} 
In practice, the procedures (\ref{fnls_311}) and (\ref{fnls_312}) must be applied in the form of some discrete version. Due to the localized character of the profiles, this is typically done by discretizing the equations on a long enough interval $(-l,l)$ with periodic boundary conditions and corresponding approximations of the linear and nonlinear operators. By way of illustration, in the case of (\ref{fnls_2327}), the discretization would have has the form
\begin{eqnarray}
Q_{h}z_{h}=G(z_{h}),\label{fnls_313}
\end{eqnarray}
where $z_{h}=(v_{h},w_{h}),$ $v_{h}=(v_{h,0},\ldots,v_{h,N-1})^{T},$ $w_{h}=(w_{h,0},\ldots,w_{h,N-1})^{T}$ denote the approximations to the values of the solution $v$ and $w$, respectively, at the grid points $x_{j}=-l+jh,$ $j=0,\ldots,N-1,$ $h=2l/N$. It is assumed that $N$ is even and $v_{h},w_{h}$ are extended as periodic functions $v_{h}=(v_{h,j})_{j\in\mathbb{Z}}$,$w_{h}=(w_{h,j})_{j\in\mathbb{Z}}$, with $v_{h,j+N}=v_{h,j}, w_{h,j+N}=w_{h,j}$, defined on a extended uniform grid $x_{j}=-l+jh, j\in\mathbb{Z}$. On the other hand, $Q_{h}$ and $G_{h}$ are, respectively, some discretizations of the operator $Q$ and the nonlinear term $G$. The discretization of (\ref{fnls22_5}) can be made in a similar way. Then, the Petviashvili method is applied to the discrete system (\ref{fnls_313}) using the corresponding Euclidean inner product.

The nonlocal character of the linear operators in (\ref{fnls22_2}) and (\ref{fnls22_5}) suggests to consider a Fourier collocation approximation in (\ref{fnls_313}) (and in the corresponding discretization of (\ref{fnls22_5})) and to implement the Petviashvili iterations in the Fourier space, that is, for the discrete Fourier coefficients of the approximations. The resulting methods were experimentally shown to be convergent (see the computations in section \ref{sec32}) and, in order to accelerate the convergence, they were performed along with the Minimal Polynomial Extrapolation (MPE) technique, (cf. e.~g. \cite{sidi,AlvarezD2016} and references therein). A brief description for the case of (\ref{fnls_311}) follows. Let $\nu\geq 0$ and let $\kappa\in \mathbb{N}\cup \{0\}$ small as a number of extrapolation steps. Let $u^{[\nu]}=\Delta z^{[\nu]}=z^{[\nu+1]}-z^{[\nu]}$. We solve the minimization problem in the corresponding norm $||\cdot ||$
\begin{eqnarray}\label{fnls_314}
\min_{c_{0},\ldots,c_{\kappa-1}}||\sum_{i=0}^{\kappa-1}c_{i}u^{[\nu+i]}+u^{[\nu+\kappa]}||,
\end{eqnarray}
take $c_{\kappa}=1$ and compute the extrapolation
\begin{eqnarray}\label{fnls_315}
z_{\nu,\kappa}=\sum_{i=0}^{\kappa}\gamma_{i}z^{[\nu+i]},\quad \gamma_{i}=\frac{c_{i}}{\sum_{j=0}^{\kappa}c_{j}},\quad i=0,\ldots,\kappa.
\end{eqnarray}
where we assume $\sum_{i=0}^{\kappa}c_{i}\neq 0$. The quantity $mw=\kappa+1$ is called the width of extrapolation. For a discussion on the choice of $\kappa$, the implementation of (\ref{fnls_314}), (\ref{fnls_315}), and convergence results, see e.~g. \cite{sidifs,smithfs,sidi}. (See also \cite{AlvarezD2016} for the application of acceleration techniques to compute solitary-wave profiles.)
\subsection{Accuracy of the approximate waves}
\label{sec32}
In this section some numerical experiments are shown to illustrate the accuracy of the procedure. This is controlled by iterating while the corresponding Euclidean norm of discrete versions of the residuals $Qz^{[\nu]}-G(z^{[\nu]})$ and $(M+a)\rho^{[\nu]}-(\rho^{[\nu]})^{2\sigma+1}$ are above some prefixed tolerance.

\begin{figure}[htbp]
\centering
\centering
\subfigure[]
{\includegraphics[width=6.2cm]{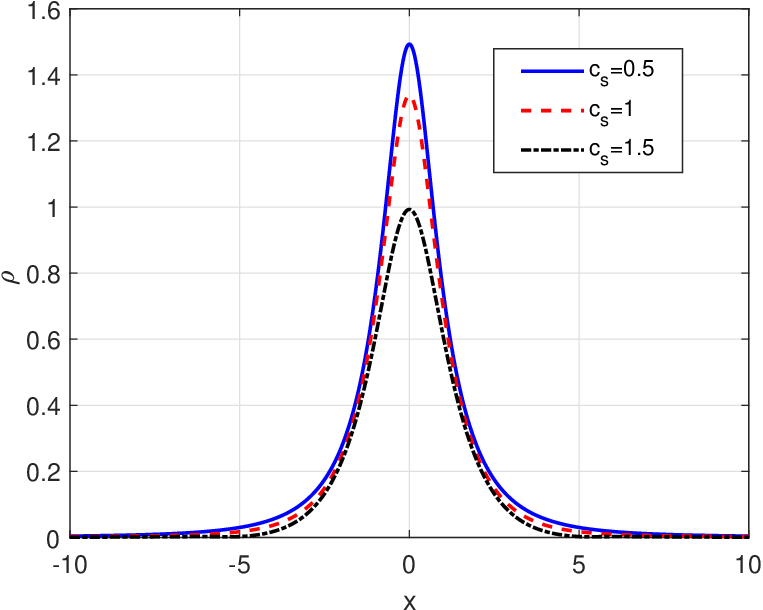}}
\subfigure[]
{\includegraphics[width=6.2cm]{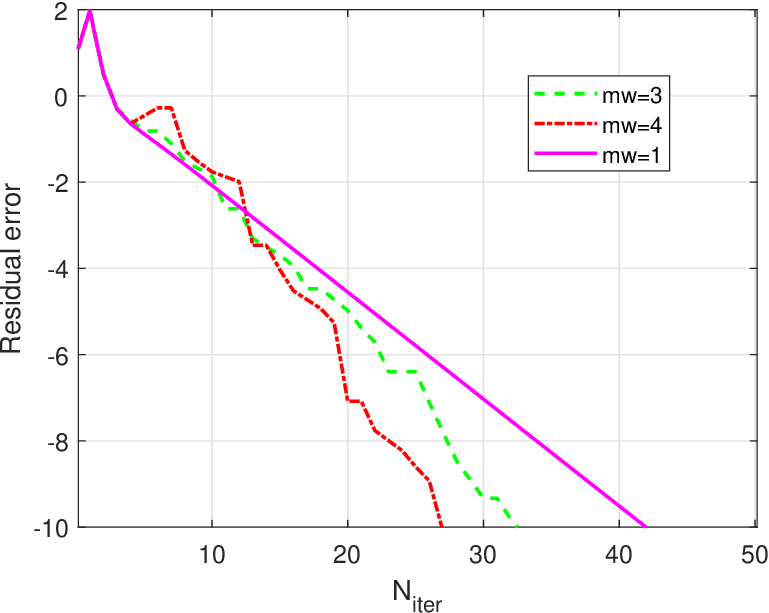}}
\caption{Numerical generation of solitary waves. Iteration (\ref{fnls_312}) with $\sigma=1, s=3/4$ y $\lambda_0^1=1$. (a) $\rho$ numerical profiles for several speeds; (b) residual error vs. number of iterations for the case $c_{s}=1$ and several values of the width of extrapolation $mw$.}
\label{NR_ADfig1}
\end{figure}
For these experiments of accuracy, we focus on the iteration (\ref{fnls_312}). Figure \ref{NR_ADfig1}(a) shows the approximate profile corresponding to $\sigma=1, s=3/4,\lambda_0^1=1$ and several speeds $c_{s}=\lambda_{0}^{2}$ below the limit $c(\lambda_{0}^{1})\approx 1.8899$ given by (\ref{fnls_236}). The convergence of the procedure and the effect of the extrapolation technique are observed from Figure \ref{NR_ADfig1}(b), which displays the Euclidean norm of the residual error as function of the number of iterations and for several values of the width $mw$ for the approximate profile obtained with $c_{s}=1$. The MPE method accelerates the convergence in the sense that the iterations required for the magnitude of the residual error to be below a fixed value are reduced. This is more clearly observed when comparing the results for $mw=1$ (no acceleration) and $mw=3$. For the experiments performed in our study, the optimal values of the width seem to be $mw=3$ and $4$. Larger values do not provide a relevant improvement, compared to the computational time required by the extrapolation (cf. \cite{smithfs,AlvarezD2016} for a discussion on the choice of $mw$).

\begin{figure}[htbp]
\centering
\centering
\subfigure[]
{\includegraphics[width=6.2cm]{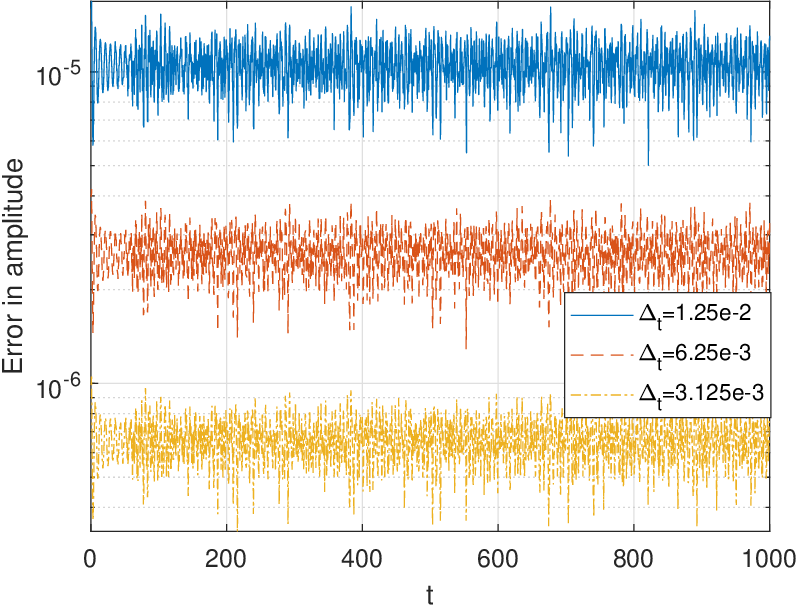}}
\subfigure[]
{\includegraphics[width=6.2cm]{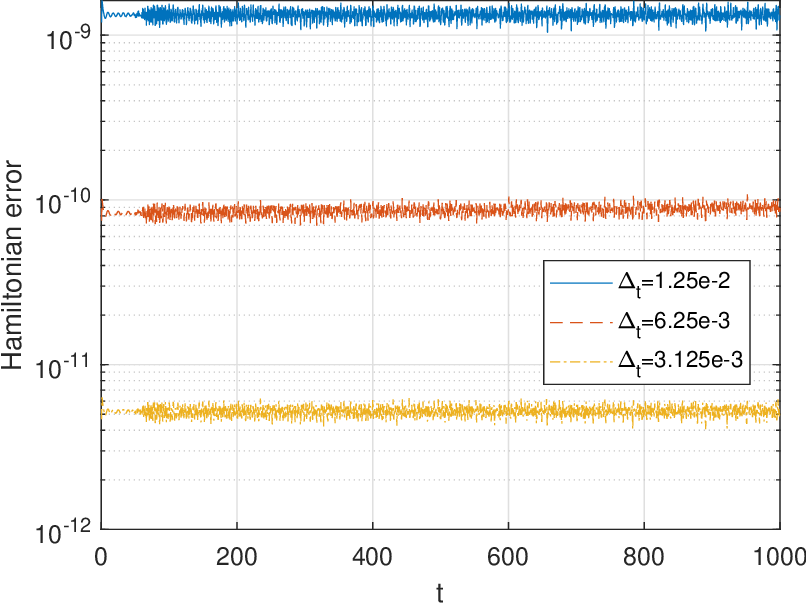}}
\caption{Numerical approximation of  (\ref{fnls1d}) with $\sigma=1, s=3/4$, initial condition given by the approximate profile from $\lambda_0^1=1, \lambda_{0}^{2}=1$. (a) Evolution of the error in the amplitude; (b) Hamiltonian error as function of time.}
\label{NR_ADfig2}
\end{figure}
A second experiment confirming the accuracy of the iteration is concerned  with the solitary-wave character of the approximate profiles. This consists of approximating the periodic ivp of  (\ref{fnls1d}) on a long enough interval with a fully discrete scheme based on a Fourier collocation spectral method for the approximation in space and the implicit midpoint rule as time integrator. The initial condition at the collocation points $x$ is of the form $\rho(x) e^{i\theta(x)}$, where $\rho$ is obtained from (\ref{fnls_312}) and $\theta=Ax$ with $A$ given by Lemma \ref{lemmaNR3} from the choice of $\lambda_{0}^{2}$. Then the evolution of the corresponding numerical solution is monitored from several points of view. Two of them are illustrated in Figure \ref{NR_ADfig2}. On the left we show the evolution of the error between the amplitude of the numerical solution and that of the initial condition, with spatial stepsize $h=6.25\times 10^{-2}$ and for several temporal stepsizes $\Delta t$. Note that the errors remain small and do not grow with time. (The same experiment but with respect to the speed $\lambda_{0}^{2}$ was also made, with similar results. The computation of these values was standard, cf. e.~g. \cite{AlvarezD2016} and references therein.) On the other hand, in Figure \ref{NR_ADfig2}(b), the evolution of the error with respect to the corresponding discrete version of the Hamiltonian (\ref{fnls3c}) shows a high level of preservation of the energy. (A similar behaviour was observed when computing the corresponding errors with respect to the quantities (\ref{fnls3a}) and (\ref{fnls3b}).)
\subsection{Additional properties of the waves}
\label{sec33}
The numerical experiments of section \ref{sec32} give confidence on the accuracy of the computed profiles and enable to study numerically additional properties of the waves.
\begin{figure}[htbp]
\centering
\centering
\subfigure[]
{\includegraphics[width=6.2cm]{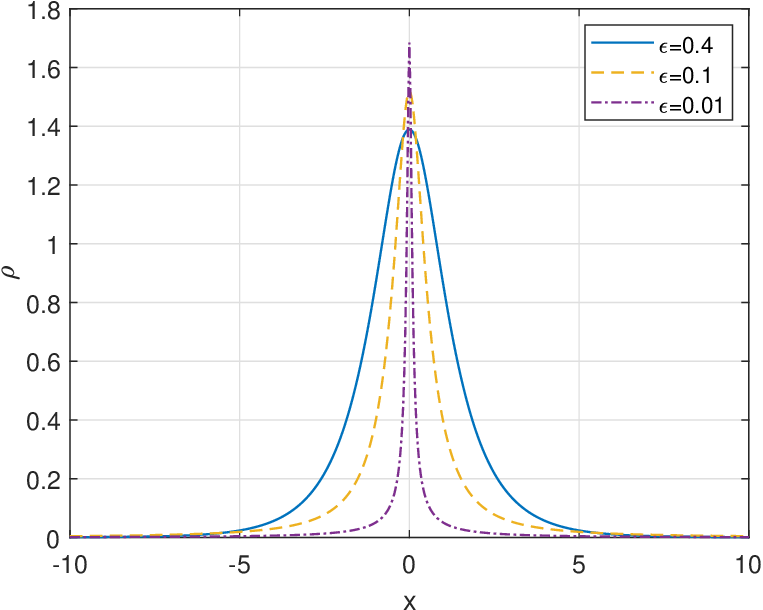}}
\subfigure[]
{\includegraphics[width=6.2cm]{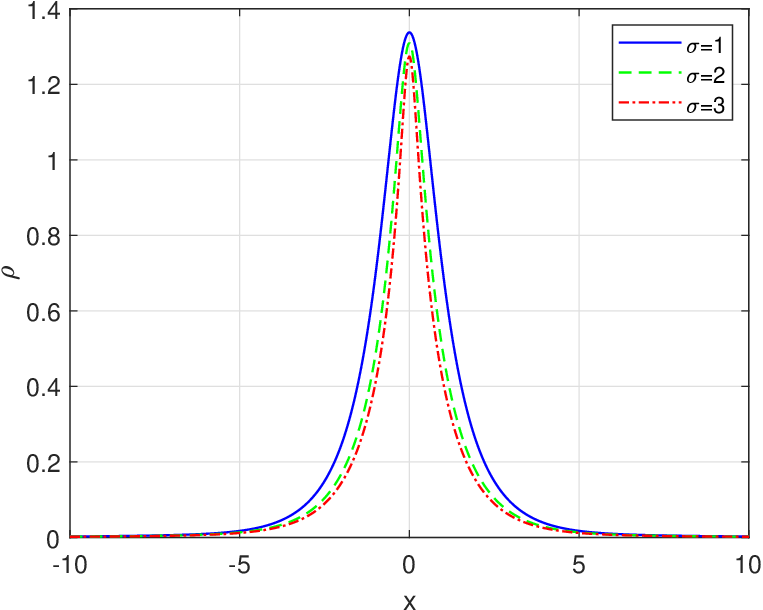}}
\caption{(a) Approximate $\rho$ profiles for $\sigma=1, \lambda_{0}^{1}=1,\lambda_{0}^{2}=0.75, s=0.5+\epsilon$ for several values of $\epsilon$. (b) Approximate $\rho$ profiles for $ s=3/4, \lambda_{0}^{1}=1, \lambda_{0}^{2}=1$ for several values of $\sigma$.}
\label{NR_ADfig3}
\end{figure}
The first experiments illustrate the influence of the fractional parameter $s$ and the parameter of nonlinearity $\sigma$ in the generation of the profiles. (Recall that both appear in the formula (\ref{fnls_236}) of the limiting value of the speed $c(\lambda_{0}^{1})$.) Figure \ref{NR_ADfig3}(a) shows the approximations of $\rho$ obtained from (\ref{fnls_312}) with  $\sigma=1, \lambda_{0}^{1}=1,\lambda_{0}^{2}=0.75, s=0.5+\epsilon$ for several values of $\epsilon$. (The value of the speed $\lambda_{0}^{2}$ always satisfies (\ref{fnls_236b}).) We observe that close to the limiting value $s=1/2$, the profiles are narrower, becoming smoother as $\epsilon$ is increasing. On the other hand, in Figure \ref{NR_ADfig3}(b) approximate profiles computed with  $ s=3/4, \lambda_{0}^{1}=1, \lambda_{0}^{2}=1$ and several values of $\sigma$ are compared. Note that the amplitude seems to decrease as $\sigma$ grows. 

\begin{figure}[htbp]
\centering
\centering
\subfigure[]
{\includegraphics[width=6.2cm]{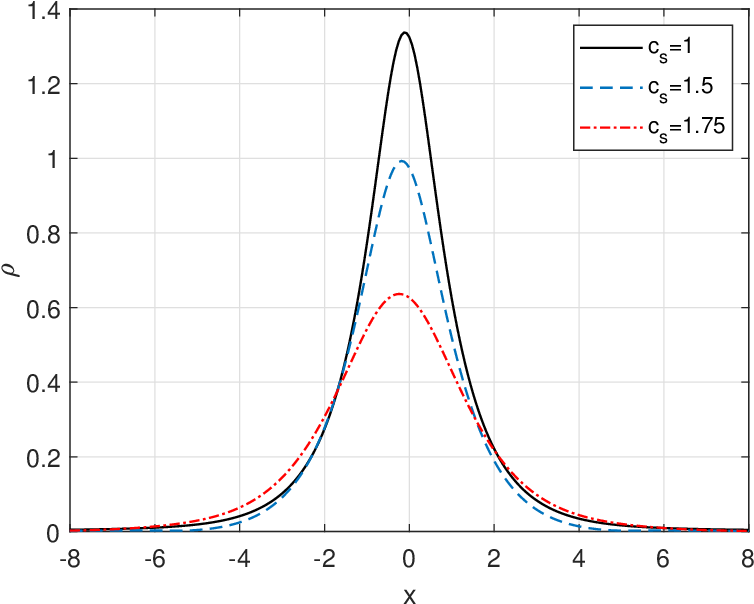}}
\subfigure[]
{\includegraphics[width=6.2cm]{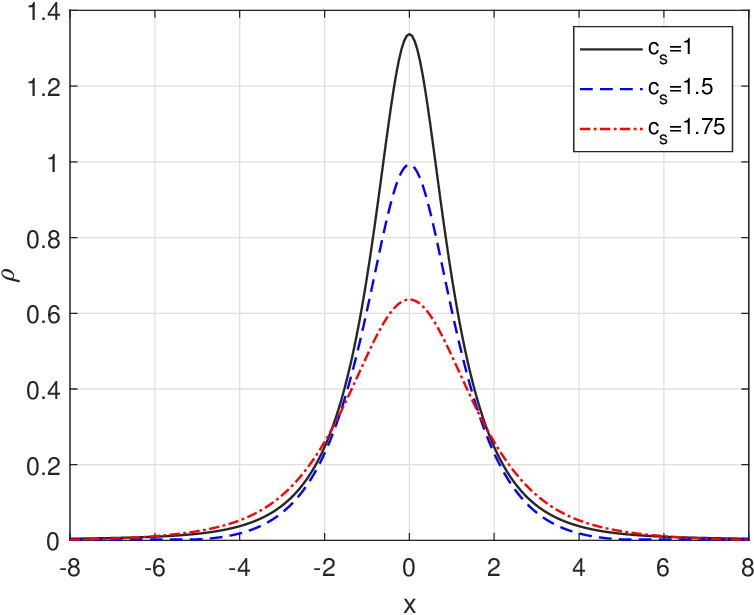}}
\caption{Approximate profiles for $\sigma=1, s=3/4, \lambda_{0}^{1}=1$. (a) $\rho=\sqrt{v^{2}+w^{2}}$ with $(v,w)$ solution of (\ref{fnls_311}) with initial iteration (\ref{fnls_331}) and $\theta(x)=x^{2}$;  (b) solution $\rho$ of (\ref{fnls_312}) with initial iteration (\ref{fnls_331}) and $\theta(x)=Ax$.}
\label{NR_ADfig4}
\end{figure}

On the other hand, the even character of the profiles does not seem to be guaranteed in general. This is confirmed by Figure \ref{NR_ADfig4}(a), which shows the magnitude $\rho=\sqrt{v^{2}+w^{2}}$ with $(v,w)$ solution of (\ref{fnls_311}) with initial iteration
\begin{eqnarray}
v^{[0]}(x)={\rm sech}(x)\cos{\theta(x)},\quad 
w^{[0]}(x)={\rm sech}(x)\sin{\theta(x)},\label{fnls_331}
\end{eqnarray}
and $\theta(x)=x^{2}$. The profiles can be compared with those shown in Figure \ref{NR_ADfig4}(b), corresponding to (\ref{fnls_312}) and initial iteration (\ref{fnls_331}) with $\theta(x)=Ax$, $A$  from Lemma \ref{lemmaNR3}. In this last case, recall that the existence of an even solution is ensured by the arguments explained in section \ref{sec231}. We observe that all the experiments performed for the preparation of this paper gave even approximate profiles. This fact and the results of section \ref{sec231} would imply that uniqueness of solution of (\ref{fnls22_5}) up to translations appears likely.

\begin{figure}[htbp]
\centering
\centering
\subfigure[]
{\includegraphics[width=6.2cm]{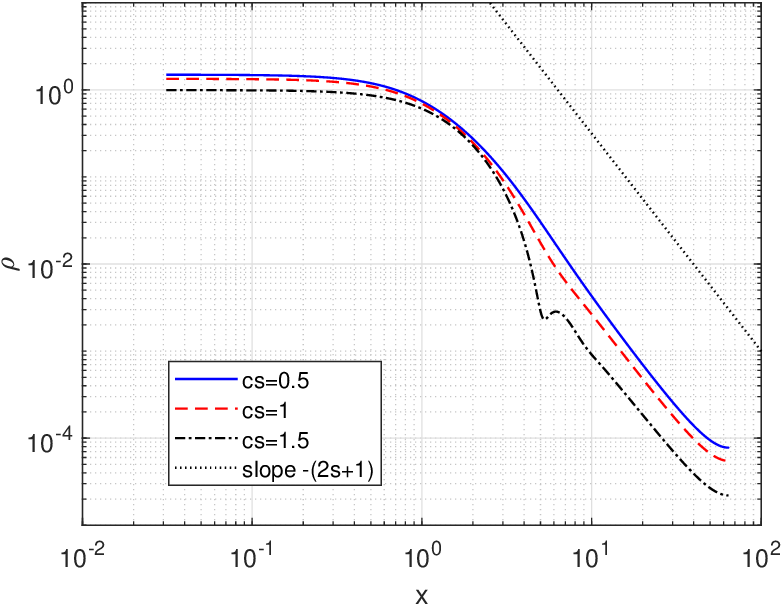}}
\subfigure[]
{\includegraphics[width=6.2cm]{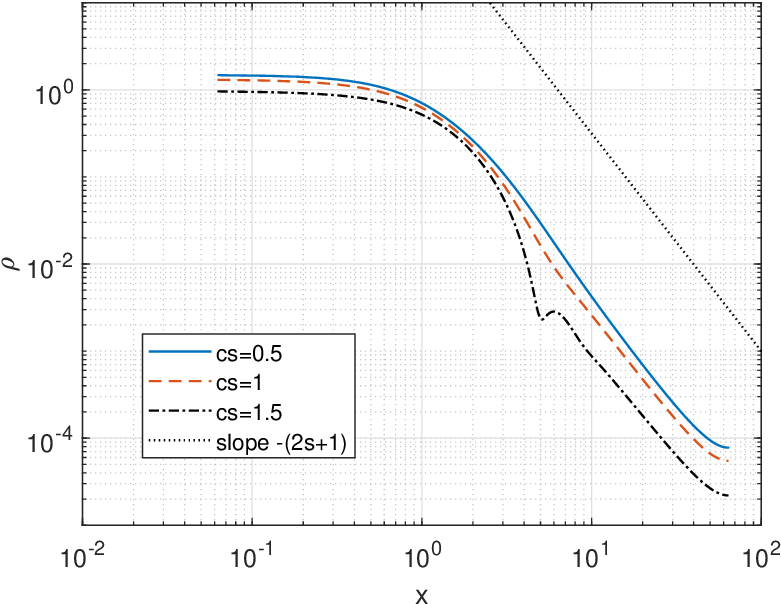}}
\caption{Approximate profiles for $\sigma=1, s=3/4, \lambda_{0}^{1}=1$ in log-log scale.  (a) Solution $\rho$ of (\ref{fnls_312}) with initial iteration (\ref{fnls_331}) and $\theta(x)=Ax$ (cf. Figure \ref{NR_ADfig1}(a)); (b) $\rho=\sqrt{v^{2}+w^{2}}$ with $(v,w)$ solution of (\ref{fnls_311}) with initial iteration (\ref{fnls_331}) and $\theta(x)=x^{2}$ (cf. Figure \ref{NR_ADfig4}(a)). }
\label{NR_ADfig5}
\end{figure}

We now complete the study of the asymptotic decay of the waves with the following experiments. Recall that in the case of (\ref{fnls22_5}), Theorem \ref{theor27} proves that the solutions $\rho(x)$ decay as $1/|x|^{2s+1}, |x|\rightarrow\infty$, cf. (\ref{fnls_2328}). This is checked in Figure \ref{NR_ADfig5}(a). This shows, in log-log scale, the numerical profiles corresponding to panel (a) of Figure \ref{NR_ADfig1}. We considered several values of the length of the interval ($l=64$ in the figure) and compared the slopes of the resulting lines with that of the dashed line, confirming a decay like $1/|x|^{2s+1}$. (Other experiments, not shown here, were made with other values of $s$.) Figure \ref{NR_ADfig5}(b) corresponds to the same experiment but for the case of the iteration (\ref{fnls_311}) with initial iteration (\ref{fnls_331}) and $\theta(x)=x^{2}$ (cf. Figure \ref{NR_ADfig4}(a)). The results suggest that Theorem \ref{theor27} might be extended to the solutions of (\ref{fnls22_2}).
\begin{figure}[htbp]
\centering
\centering
\subfigure[]
{\includegraphics[width=6.2cm]{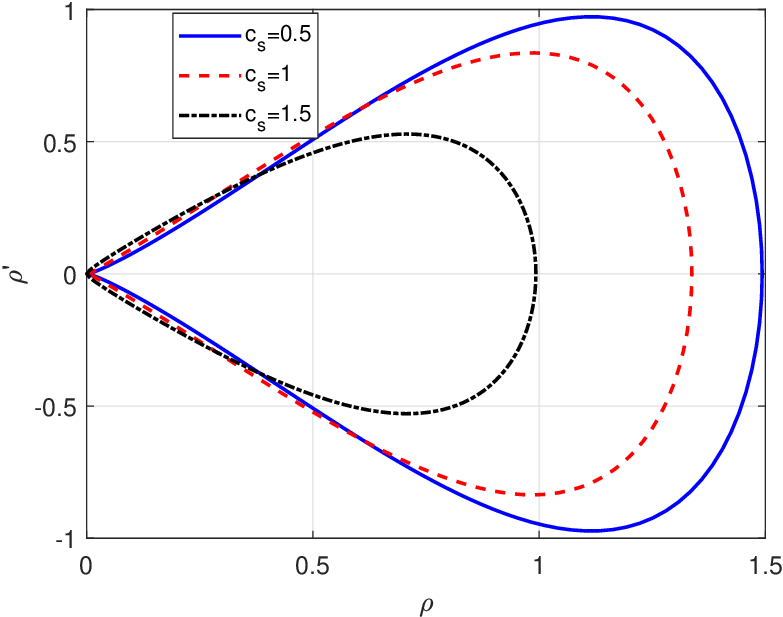}}
\subfigure[]
{\includegraphics[width=6.2cm]{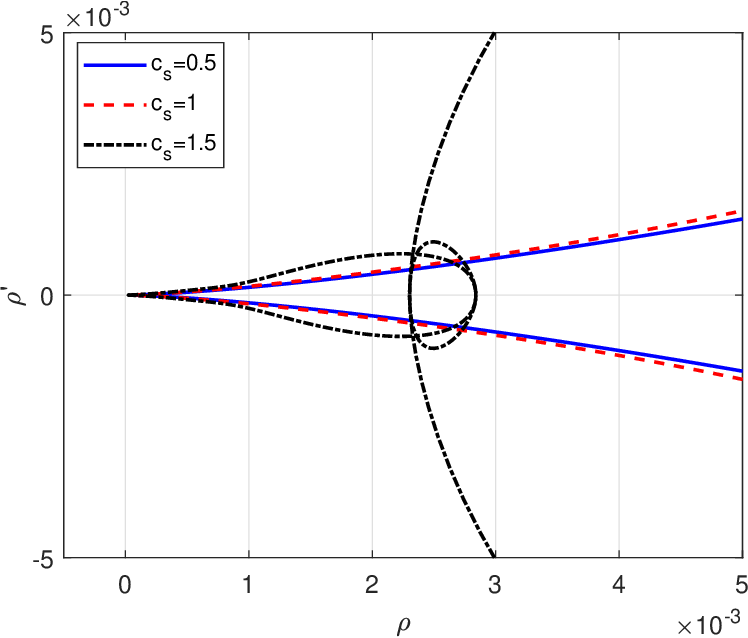}}
\subfigure[]
{\includegraphics[width=6.2cm]{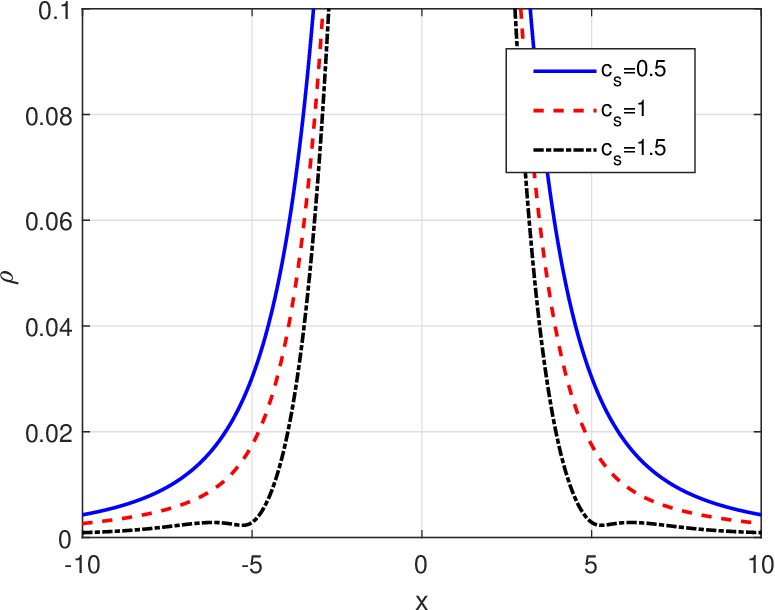}}
\caption{Approximate profiles for $\sigma=1, s=3/4, \lambda_{0}^{1}=1$. (a) Phase plot of the profiles in Figure \ref{NR_ADfig1}(a)); (b) magnification of (a); (c) magnification of Figure \ref{NR_ADfig1}(a).}
\label{NR_ADfig6}
\end{figure}

Some additional information suggested by Figure \ref{NR_ADfig5} is concerned with the behaviour of the decay of the waves. Figure \ref{NR_ADfig6}(a) shows the phase plot of the profiles displayed in Figure \ref{NR_ADfig1}. Note first that the way how the waves, as homoclinic orbits, approach the origin at infinity confirms the algebraic decay. On the other hand, the magnifications shown in Figures \ref{NR_ADfig6}(b) and \ref{NR_ADfig6}(c) suggest that the decay is not monotone, and the waves with speed close to the limiting value start to develop symmetric oscillations, breaking the decreasing behaviour. This phenomenon was also observed in experiments for the general case (\ref{fnls_311}).

\begin{figure}[htbp]
\centering
\centering
\subfigure[]
{\includegraphics[width=6.2cm]{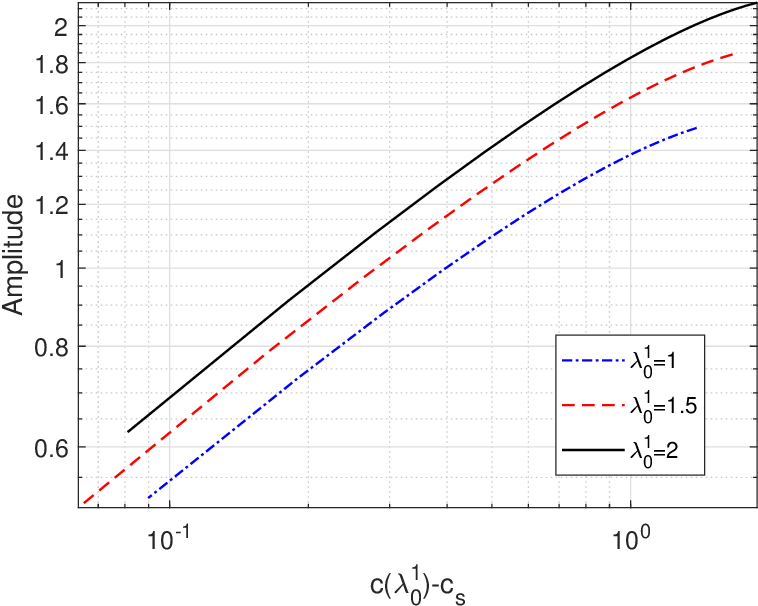}}
\subfigure[]
{\includegraphics[width=6.2cm]{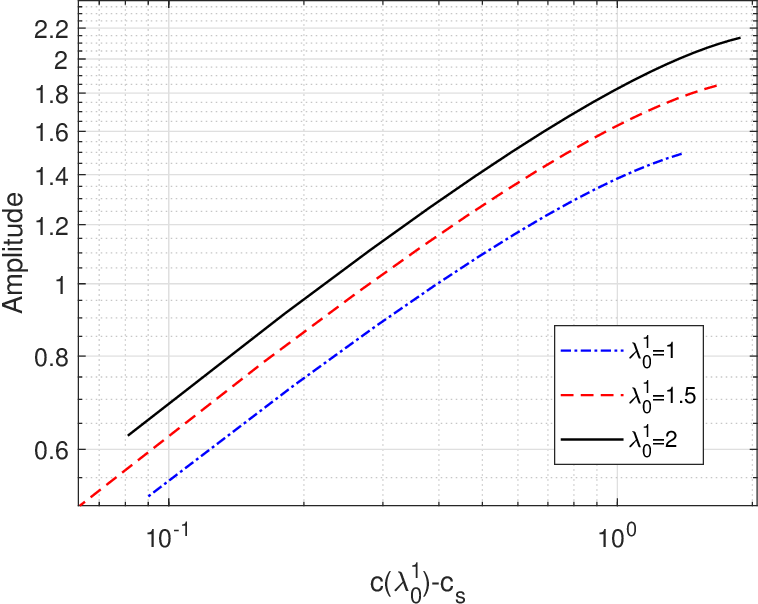}}
\caption{Speed-amplitude relation. (a) Solution $\rho$ of (\ref{fnls_312}) with initial iteration (\ref{fnls_331}) and $\theta(x)=Ax$ (cf. Figure \ref{NR_ADfig1}(a)); (b) $\rho=\sqrt{v^{2}+w^{2}}$ with $(v,w)$ solution of (\ref{fnls_311}) with initial iteration (\ref{fnls_331}) and $\theta(x)=x^{2}$.}
\label{NR_ADfig7}
\end{figure}

A final property studied here concerns the speed-amplitude relation. This is illustrated in Figure \ref{NR_ADfig7}(a), (b). In the first case, for several values of the limiting value of the speed $c(\lambda_{0}^{1})$ in (\ref{fnls_236}), the amplitude of several profiles with speeds $c_{s}$ satisfying (\ref{fnls_236b}) is computed. This leads to a nonlinear increasing relation between the amplitude and the difference $c(\lambda_{0}^{1})-c_{s}$. A similar behaviour is observed in the case of the iteration (\ref{fnls_311}), as suggested by  Figure \ref{NR_ADfig7}(b). This was already observed in Figure \ref{NR_ADfig5}.




\section{Concluding remarks}
\label{sec4}
The present paper is concerned with the existence and properties of solitary-wave solutions of the fractional nonlinear Schr\"{o}dinger equation (\ref{fnls1}). Following the approach developed in \cite{DuranS2000} for the classical NLS and its soliton solutions, here we look for solitary wave solutions as equilibria of the Hamiltonian relative to fixed values of two additional conserved quantities, related to the symmetry group of translations and phase rotations. The existence of solution of the relative equilibrium condition, written as a real, nonlocal, differential system (\ref{fnls22_2}) is established from the application of the Concentration-Compactness theory of \cite{Lions}. The existence is obtained for a range of speeds below a limiting value (speed of sound) which depends on the fractional parameter $s\in (1/2,1]$ and the parameter determining the nonlinearity $\sigma>0$. The regularity of the waves is derived as corollary and the asymptotic decay at infinity, as $1/|x|^{2s+1}$, is proved for a particular subfamily of complex profiles with linear phase function, whose existence was previously derived in \cite{HongS2017}.

Other properties of the waves are studied computationally. To this end, an iterative procedure for the approximation of the profiles, based on the Petviashvili iteration, \cite{Petv1976}, with extrapolation is introduced and its good performance is checked with several numerical experiments. The accuracy of the computed profiles is used to make a second group of experiments with the aim at investigating additional features of the waves. The main conclusions are the following:
\begin{itemize}
\item The solitary-wave profiles are not even in general; only in the particular case of the subfamily derived in \cite{HongS2017} the experiments suggest an even character and uniqueness under translations and rotations.
\item The asymptotic decay of the waves seems to be, in general, like $1/|x|^{2s+1}$ as $|x|\rightarrow\infty$, extending the result proved for the subfamily introduced in \cite{HongS2017}, cf. (\ref{fnls_2328}).
\item The waves do not decay to zero at infinity in a monotone way; the profiles may develop small oscillations which break the decreasing behaviour.
\item The amplitude of the waves is a nonlinear, increading function of the speed, relative to the corresponding speed of sound.
\end{itemize}
The results, both theoretical and computational, obtained in the present work will be especially helpful in a companion paper, \cite{NRAD2}, which will examine the numerical analysis of (\ref{fnls1}) and the dynamics of the solitary-wave solutions.

\section*{Acknowledgments}
A. D. is supported by the Junta de Castilla y Le\'on and FEDER funds (EU) under Research Grant VA193P20.

\bibliographystyle{plain}

\begin{thebibliography}{ablowitzm}
\bibitem{AblowitzP} M. Ablowitz, B. Prinari, Nonlinear Schr\"{o}dinger systems: continuous and discrete (2008). Scholarpedia 3(8): 5561.
\bibitem{Adams1975} Adams, R. A., Sobolev Spaces, Academic Press, New York, 1975.
\bibitem{Albert} J. P. Albert, Positivity properties and stability of solitary-wave solutions of model equations for long waves, Comm. PDE, 17, 1992, 1-22.
\bibitem{ABS1997} Albert, J. P., Bona, J. L., Saut, J. C., Model equations for waves in stratified fluids, Proc. Roy. Soc. A, 453, 1997, 1233-1260.
\bibitem{AlvarezD2014} \'Alvarez, J., Durán, A., Petviashvili type methods for traveling wave computations: I. Analysis of convergence, J. Comput. Appl. Math., 266, 2014, 39-51.
\bibitem{AlvarezD2016} \'Alvarez, J., Durán, A., Petviashvili type methods for traveling wave computations:
II. Acceleration with vector extrapolation methods, Math. Comput. Simul., 123, 2016, 19-36.
\bibitem{AnguloS2020} J. Angulo-Pava, J.-C. Saut, Existence of solitary wave solutions for internal waves in two-layer systems, Quart. Appl. Math. 78 (2020), 75-105.
\bibitem{AntoineTZ2016} Antoine, X., Tang, Q., Zhang, Y., On the ground states and dynamics of space fractional nonlinear Schr\"{o}dinger/Gross–Pitaevskii equations with rotation term and nonlocal nonlinear interactions, J. Comput. Phys, 325, 2016, 74-97.
\bibitem{ChenB} H. Chen, J. L. Bona, Existence and asymptotic properties of solitary-wave solutions of Benjamin-type equations, Adv. Diff. Eq., 3(1), 1998, 51-84.
\bibitem{ChoHHO2013} Cho, Y., Hajaiej, H., Hwang, G., Ozawa, T., On the Cauchy problem of fractional Schr\"{o}dinger equation with Hartree type nonlinearity Funkcial. Ekvac. 56, 2013, 193–224.
\bibitem{ChoHKL2015} Cho, Y., Hwang, G., Kwon, S., Lee, S., Well-posedness and ill-posedness for the cubic fractional Schr\"{o}dinger equations Disc. Cont. Dyn. Syst., 35, 2015, 2863-80.
\bibitem{ChoHKL2015b} Cho, Y., Hwang, G., Kwon, S., Lee, S., On the finite time blowup for mass-critical Hartree equations, Proc. Roy. Soc. Edinburgh A, 145(3), 2015, 467-479.
\bibitem{ChoHHO2014} Cho, Y., Hwang, G., Hajaiej, H., Ozawa, T., On the orbital stability of fractional Schr\"{o}dinger equations, Comm. Pure Appl. Anal., 13, 2014, 1267–1282
\bibitem{Driscoll2002} Driscoll, T. A., A composite Runge–Kutta Method for the spectral solution of semilinear PDEs. J. Comput. Phys. 182, 2002, 357–367.
\bibitem{DuoLZ} Duo, S., Lakoba, T. I., Zhang, Y., Analytical and numerical study of plane wave dynamics in the fractional nonlinear Schr\"{o}dinger equation, preprint.
\bibitem{DuoZ2016} Duo, S., Zhang, Y., Mass-conservative Fourier spectral methods for solving the fractional nonlinear Schr\"{o}dinger equation, Comput. Math. Appl., 71, 2016, 2257-2271.
\bibitem{NRAD2} Dur\'an, A., Reguera, N., Solitary-wave solutions of the fractional nonlinear Schr\"{o}dinger equation. II.  Numerical solution and dynamics. To appear.
\bibitem{DuranS2000} Dur\'an, A., Sanz-Serna, J. M.,The numerical integration of relative equilibrium solutions. The nonlinear Schr\"{o}dinger equation, IMA J. Numer. Anal., 20, 2000, 235-261.
\bibitem{FrankL} Frank, R.L., Lenzmann, E., Uniqueness of nonlinear ground states for fractional Laplacians in $\mathbb{R}$, Acta Math 210, 261–318 (2013).
\bibitem{FrohlichJL2007} Fröhlich, J., Jonsson, B., Lenzmann, E., Boson stars as solitary waves. Comm. Math. Phys., 274, 2007, 1-30.
\bibitem{GuoH} Guo, B., Huang, D., Existence and stability of standing waves for nonlinear fractional Schr\"{o}dinger equations. J. Math. Phys. 53, 012, 083702, 12pp.
\bibitem{GuoH2013} Guo, B., Huo, Z., Well-posedness for the nonlinear fractional Schr\"{o}dinger equation and inviscid limit behavior of solution for the fractional Ginzburg–Landau equation. Fractional Calc. Appl. Anal. 16, 2013, 226-242.
\bibitem{GuoSWZ2013} Guo, Z., Sire, Y., Wang, Y., Zhao, L., On the energy-critical fractional Schr\"{o}dinger equation in the radial case, 2013, (http://arxiv.org/abs/1310.6816).
\bibitem{HongS2015} Hong, Y., Sire, Y., On the fractional Schr\"{o}dinger equation in Sobolev spaces Commun. Pure Appl. Anal. 14, 2015, 2265–2282.
\bibitem{HongS2017} Hong, Y., Sire, Y., A new class of traveling solitons for cubic fractional nonlinear Schr\"{o}dinger equations, Nonlinearity, 30, 2017, 1262-1286.
\bibitem{IonescuP2014} Ionescu, A. D., Pusateri, F., Nonlinear fractional Schr\"{o}dinger equations in one dimension, J. Fun. Anal., 266(1), 2014, 139-176.
\bibitem{KP} Kato, T., Ponce, G., Commutator estimates and the {E}uler and {N}avier-{S}tokes equations, Comm. Pure Appl. Math., 41(7), 1988, 891-907.
\bibitem{KPV} Kenig, C. E.,  Ponce, G.,  Vega, L., Well-posedness of the initial value problem for the Korteweg-de Vries equation, J. Amer. Math. Soc., 4(2), 1991, 323-347.
\bibitem{Klein2008} Klein, C., Fourth order time-stepping for low dispersion Korteweg–de Vries and nonlinear Schr\"{o}dinger equations. Electronic Trans. Num. Anal. 29, 2008, 116-135.
\bibitem{KleinSM2014} Klein, C., Sparber, C., Markowich, P., Numerical study of fractional nonlinear Schr\"{o}dinger equations, Proc. Roy. Soc. A 470, 2014, 20140364.
\bibitem{KirkpatrickZ2016} Kirkpatrick, K., Zhang, Y., Fractional Schr\"{o}dinger dynamics and decoherence, Physica D, 332(1), 2016, 41-54.
\bibitem{Laskin2000} Laskin, N., Fractional quantum mechanics and Lévy path integrals. Phys. Lett. A 268, 2000, 298-305.
\bibitem{Laskin2002} Laskin, N., Fractional Schr\"{o}dinger equation, Phys. Rev. E., 66, 2002, 056108.
\bibitem{Laskin2011} Laskin, N., Principles of Fractional Quantum Mechanics, Fractional Dynamics: 393-427, 2011 (http://www.arxiv.org/abs/1009.5533).
\bibitem{Lenzmann2007} Lenzmann, E., Well-posedness for semi-relativistic Hartree equations of critical type, Math. Phys. Anal. Geom. 10, 2007, 43-64.
\bibitem{LiHW2017}  Li, M., Huang, Ch., Wang, N., Galerkin finite element method for the nonlinear fractional Ginzburg-Landau equation, Appl. Numer. Math., 118, 2017, 131-149.
\bibitem{Lions} P. L. Lions. (1984), The concentration-compactness principle in the calculus of variations. The locally compact case. Part I and Part II. Ann. Inst. Henri Poincar\'e Sect A (N.S.) 1, pp. 109-145 and pp. 223-283.
\bibitem{ObrechtS2015} Obrecht, C., Saut, J-C., Remarks on the full dispersion Davey-Stewartson system. Comm. Pure Appl. Anal., 14, 2015, 1547-1561.
\bibitem{Olver} P. J. Olver, Applications of Lie Groups to Differential Equations, Springer Verlag, 1993.
\bibitem{pelinovskys} {D. E. Pelinovsky and Y. A.
Stepanyants}, { Convergence of Petviashvili's iteration method for
numerical approximation of stationary solutions of nonlinear wave
equations}, { SIAM J.~Numer.\ Anal.} { 42} (2004) 1110-1127.
\bibitem{Petv1976} { V. I. Petviashvili} { Equation of
an extraordinary soliton}, { Soviet J.~Plasma Phys.} { 2} (1976)
257-258.
\bibitem{sidi} A. Sidi, {Vector Extrapolation Methods with Applications}, SIAM Philadelphia, 2017.
\bibitem{sidifs} A. Sidi, W. F. Ford, D. A. Smith, Acceleration of convergence of vector sequences, SIAM J. Numer. Anal., 23 (1986) 178-196.
\bibitem{smithfs} D. A. Smith, W. F. Ford, A. Sidi, Extrapolation methods for vector sequences, SIAM Rev., 29 (1987) 199-233.
\bibitem{SulemS1999} Sulem, C., Sulem, PL., The nonlinear Schr\"{o}dinger equation: self-focusing and wave collapse, Springer Series in Mathematical Sciences, vol. 139, Springer, Berlin, 1999.
\bibitem{WangH2015} Wang, P., Huang, C., An energy conservative difference scheme for the nonlinear fractional Schr\"{o}dinger equations, J. Comput. Phys., 293, 2015 238-251.
\bibitem{WangXY2014} Wang, D., Xiao, A., Yang, W., A linearly implicit conservative difference scheme for the space fractional coupled nonlinear Schr\"{o}dinger equations, J. Comput. Phys., 272, 2014, 644-655.
\bibitem{Weinstein1987} M. I. Weinstein, Existence and dynamic stability of solitary wave solutions of equations arising in long wave propagation, Commun. Partial Diff. Eq., 12 (10), 1987, 1133-1173.
\end{thebibliography}

\bigskip

\end{document}